\documentclass[11pt,leqno]{amsart}
\usepackage{amsmath,amssymb,amsfonts,yhmath,mathrsfs}
\usepackage{delarray}
\usepackage{floatflt,tikz,pgf,pgfplots}
\usepackage{hyperref}
\usepackage{changes}

\numberwithin{equation}{section}

\newcommand{\eps}{\varepsilon}

\newcommand{\om}{\omega}
\newcommand{\Om}{\Omega}

\newcommand{\vth}{\vartheta}
\renewcommand{\th}{\theta}

\newcommand{\D}{\mathbb{D}}
\newcommand{\z}{\zeta}

\def\erre{\mathbb{R}}
\def\enne{\mathbb{N}}
\def\zzz{\mathbb{Z}}

\newcommand\ci{\mathbb{C}}
\newcommand\R{\mathbb{R}}
\newcommand\C{\mathbb{C}}
\def\ee{{\rm e}}

\def\arg{{\rm Arg}\,}
\newcommand\fU{{f_{U}}}
\renewcommand{\d}{\textup{d}}

\newcommand{\cC}{{\mathcal{C}}}
\newcommand{\cN}{{\mathcal{N}}}

\newcommand{\cU}{{\mathcal{U}}}
\newcommand{\cZ}{{\mathcal{Z}}}

\newcommand{\pa}{\partial}
\newcommand{\ov}{\overline}
\newcommand{\dyle}{\displaystyle}

\newtheorem{theorem}{Theorem}[section]
\newtheorem{remark}[theorem]{Remark}

\newtheorem{lemma}[theorem]{Lemma}
\newtheorem{proposition}[theorem]{Proposition}
\newtheorem{corollary}[theorem]{Corollary}

\begin{document}

\title[Generic 2D competing systems]{Generic configurations in 2D strongly competing systems}
\author[Lanzara, Montefusco, Nesi \& Spadaro]{Flavia Lanzara, 
Eugenio Montefusco, \\ Vincenzo Nesi \& Emanuele Spadaro}
\address{Dipartimento di Matematica, Sapienza Università di Roma,
piazzale Aldo Moro 5, I-00185 Roma Italy}
\date{}
\maketitle


\begin{abstract}
We study a problem modelling segregation of an arbitrary number 
of competing species in planar domains.
The solutions give rise to a well known free boundary problem 
with the domain partitioning itself into subdomains occupied 
by different species.
Several subdomains can coexist in a neighborhood of any point. However, we show that {\it generically} (in the sense of Baire category) only triple junctions appear, meaning 
that at most three subdomains meet at their common free boundary. 
Our main tools are the use of the formalism of harmonic maps 
into singular spaces and the introduction of a complex structure 
via the Hopf differential. 

\end{abstract}

\section{Introduction}

For any positive natural number \(N\geq 2\), we consider the metric space
\[
\Sigma_N := \big\{t e_i : i=1, \ldots, N,\; t\geq 0\big\}\subset\R^N,
\]
endowed with the induced geodesic distance ($e_i$ running over the standard basis).
Let $\Omega\subset\ci$ be a bounded simply connected domain. 
In the spirit of the pioneering work of Gromov and Schoen \cite{gs}, we consider the harmonic maps $U\in H^1(\Omega; \Sigma^N)$, which are the stationary points of the Dirichlet energy
\begin{equation}\label{e.dirichlet energy}
E(U):=\frac12 \int_{\Omega} |\nabla U|^2\,dx.
\end{equation}
In this paper we study the geometry of the \emph{generic} (in the sense of Baire category) harmonic maps with values into $\Sigma_N$.
To this aim, for given $N$ we introduce the set \(\cU_N(\Omega)\subset H^1(\Omega;\Sigma_N)\) of all the stationary points of \eqref{e.dirichlet energy} and we set
\begin{equation}\label{e.U}
\cU(\Omega):=\bigcup_{N\geq 2} \Big\{u_1+\cdots + u_N : (u_1,\ldots , u_N)\in \cU_N(\Omega)\Big\}
\end{equation}
and
\[
\overline{\cU(\Omega)}:=H^1\textup{-closure of } \cU(\Omega).
\]

It is well-known (see Corollary \ref{c.reg1}) that for every $U=(u_1, \ldots, u_N)\in \cU_N(\Omega)$ each component $u_j$ is locally Lipschitz continuous and, by the very definition, it is harmonic in its domains
of positivity
\[
-\Delta u_j = 0 \quad\textup{in}\quad \omega_j:=\{u_j>0\}
\qquad \forall \ j=1,..., N.
\]
The multiplicity at a point $z\in \Omega$ is defined as the positive integer
\begin{equation}\label{e.multiplicity m}
m_U(z):=\lim_{r\to 0^+}\#\{j: |\om_j\cap \D_r(z)|>0\},
\end{equation}
where $\D_r(z)$ denotes the open disk centered in \(z\) and 
radius \(r\) and $\#$ the cardinality. When $m_U(z)=m$ we call $z$ 
an \(m\)-point.
The interesting points in our analysis are those with multiplicity 
at least $3$. The \(1\)-points are the interiors of the positivity regions $\omega_j$, whereas the \(2\)-points are the locally smooth interfaces between exactly
two such regions (see \S \ref{s.3}):
\[
m(z) = 2 \quad \Longrightarrow \quad z \in 
\partial\omega_i \cap \partial\omega_j \quad i\neq j.
\]
The set of $m$-points with $m\geq 3$ in $\Omega$ is  discrete and can be regarded as the singular set of the harmonic maps $U:\Omega\to \Sigma_N$.

\medskip

The aim of this paper is to classify the structures of the singular set of the \emph{generic} harmonic maps $\cU(\Omega)$.
Leaving apart the cases $N=2, 3$ which are not relevant to the present study, a classification has been 
completed for \(N=4\) in \cite{lm1}, where the authors prove that two configurations are possible, namely  the case of a single \(4\)-point, 
or configurations with two \(3\)-points (see Figure \ref{figtype}).

\begin{figure}[h!]
\begin{center}
\begin{tikzpicture}
\draw[thick] (0,0) circle (2cm);
\draw (-2,0) -- (2,0);
\draw (0,-2) -- (0,2);
\draw (-2,0) circle (3pt);
\draw (2,0) circle (3pt);
\draw (0,-2) circle (3pt);
\draw (0,2) circle (3pt);
\filldraw[black] (0,0) circle (2pt);
\end{tikzpicture}
\qquad
\begin{tikzpicture}
\draw[thin] (0,0) circle (2cm);
 \coordinate (P1) at (0.3,0.3);
 \coordinate (P2) at (-0.3,-0.3);
\draw (-2,0) -- (P2);
\draw (0,-2) -- (P2);
\draw (2,0) -- (P1);
\draw (0,2) -- (P1);
\draw (P1) -- (P2);
\draw (-2,0) circle (3pt);
\draw (2,0) circle (3pt);
\draw (0,-2) circle (3pt);
\draw (0,2) circle (3pt);
\filldraw[black] (P1) circle (2pt);
\filldraw[black] (P2) circle (2pt);
\end{tikzpicture}
\end{center}
\caption{Four species: configurations with one 4-point 
(on the left) and two 3-points (on the right).}
\label{figtype}
\end{figure}
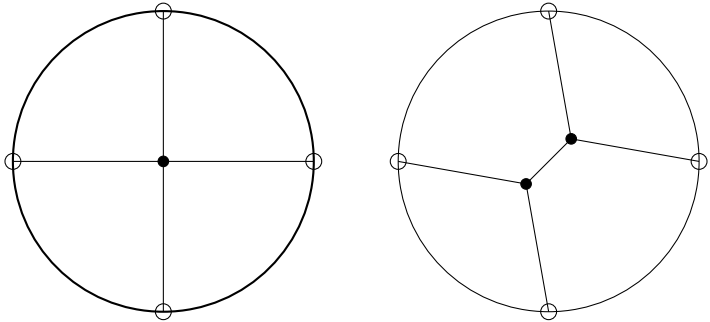

As explained in \cite{lm1}, when \(N=4\) the existence of \(4\)-points requires 
an additional condition and, therefore, it is not a stable configuration. This suggests that the 
solutions with \(4\)-points lie on a Banach manifold with finite 
co-dimension, and therefore are not generic in the space of solutions.

In this paper we show that the result obtained for \(N=4\) is paradigmatic of the general case $N> 3$. We prove that the configurations where the 
subdomains meet exclusively on \(2\)-points or \(3\)-points are residual in the sense of Baire category. A direct consequence is that the points with higher multiplicity
are unstable and can be perturbed away by small variations.

\medskip

To state our main result we introduce the set of functions in $\cU(\Omega)$ corresponding to harmonic maps $U\in \cU_N(\Omega)$ with at most triple junctions:
\begin{align}
\mathcal{U}_N^{\curlyvee}(\Omega):=
\Big\{U=(u_1, \ldots, u_N)\in \cU_N(\Omega),
 \ m_U(z)\leq 3, \ \forall \ z\in \Omega \Big\}
\end{align}
and 
\begin{align}
\mathcal{U}^{\curlyvee}(\Omega)
:=&\bigcup_{N\geq 2}
\Big\{u_1+\cdots +u_N  \ :\  (u_1, \ldots, u_N)\in \cU_N^{\curlyvee}(\Omega) \Big\} \subset  \cU(\Omega).
\end{align}

Our main result is the following.

\begin{theorem}\label{t:main}
For $\Omega\subset \ci$ a bounded simply connected domain, $\mathcal{U}^{\curlyvee}(\Omega)\subset \overline{\cU(\Omega)}$ is residual in the sense of Baire category.
\end{theorem}

\subsection{A model for competing species}
Problem \eqref{e.dirichlet energy} is linked to the following
competition-diffusion system of \(N\) differential equations
\begin{equation}\label{diffeqk}
\begin{array}\{{ll}.
-\Delta u_j=-\mu u_j \dyle\sum_{k\neq j} u_k
	& \hbox{ in }\Omega, \\
u_j\geq 0 & \hbox{ in }\Omega, \\
u_j=g_j & \hbox{ on } \pa\Omega,
\end{array} \qquad j=1,\ldots,N,
\end{equation}

This system can be viewed as a simplified model for the study of the
segregation phenomena, when one interpretes \(u_j\) in \eqref{diffeqk}
as the density of the \(j\)-th specie.
From the PDEs point of view, this is an obstacle problem with unilateral constraints $u_j\geq 0$ and Dirichlet boundary values
\begin{equation*}
g_j\geq 0, \quad g_j\, g_k=0, \quad
G:=g_1+\cdots+g_N \neq 0 \quad \hbox{ $\mathcal{H}^1$-a.e.}
\end{equation*}
The assumption $g_j\geq 0$ not identically zero is  motivated by the model.
The parameter \(\mu>0\) quantifies
the interaction strength among the populations.

In \cite{ctv2, wz} the existence and uniqueness of the solutions of
\eqref{diffeqk} are shown.
Problem \eqref{e.dirichlet energy} is the limiting case of \eqref{diffeqk} when
the parameter \(\mu\) tends to \(+\infty\), see \cite{ctv2}.
In the last decade the qualitative properties of the solutions have been the
object of an intensive study, see
\cite{ab,bz,ckl,cr,cl,clll,ctv0,ctv1,ctv2,ctv3,ctv4,dwz,lm1,lm2,sz,wz} and the
references therein.

The connection between \eqref{e.dirichlet energy} and \eqref{diffeqk} arises because, as it has been proved in \cite{ctv2}, the one parameter solutions to the
reaction-diffusion system \eqref{diffeqk} converge, as
\(\mu\to+\infty\), to the unique solution of the variational problem \eqref{e.dirichlet energy} with the same boundary conditions.

\subsection{Geometry of nodal set}
In this paper, we depart from the well-studied approach of fixing the boundary values.
We introduce a novel method to analyze the nodal set of a typical map within the space of all solutions, regardless of the possible traces of $H^1(\Omega; \Sigma_N)$-maps.

Our starting point is the work \cite{lm1}, where for \(N=4\) the possible
\(4\)-point configurations are characterized by means of the solution of
a Dirichlet problem for the Laplace equation. The necessary and sufficient
conditions on the datum which generates a \(4\)-point suggest that in the
most common configurations only \(3\)-points appear.
The geometry of the solutions $\cU(\Omega)$ for any \(N\) was already studied in \cite{lm2}. In particular, when the multiplicity of each singular
point is even, it is proved that the solution is the absolute value of the harmonic
function which assumes the value \(\sum_{j=1}^{N}(-1)^j g_j\) on \(\pa\Omega\).

The connection to the harmonic functions envisaged in \cite{lm1,lm2}
suggests to introduce a global approach based on complex analytic functions.
We consider the whole set \(\cU(\Omega)\) of harmonic maps
for varying boundary values and varying \(N\).
Following a well-known idea
long exploited in harmonic map theory (see, e.g., \cite{Schoen}), we
associate to each solution \(U\in\cU\) its Hopf differential
\[
f_U:= U_z^2 =
\textstyle{\frac14} 
\big(U_{x_1}^2-U_{x_2}^2
-2iU_{x_1}U_{x_2}\big), \qquad z=x_1+i x_2.
\]
It follows from the inner variations
of the Dirichlet energy that \(f_U\) is holomophic for every
\(U\in\mathcal{U}\) (see Proposition \ref{p.dez hol}). Moreover,
the points \(z\) with multiplicity \(m_U(z)\geq 3\) are zeros of
the holomorphic function \(f_U\) and, specifically, the multiplicity
\(3\)-points are in one to one correspondence with the simple zeros
of \(f_U\):
\[
m_U(z) = 3 \qquad \Longleftrightarrow \qquad \textup{ord}(f_U; z) =1.
\]
The problem of the genericity of multiplicity at most \(3\)-points
is therefore translated to the genericity of Hopf differentials
with simple zeros. Although holomorphic functions with zeros with
order \(1\) are clearly residual (actually an open dense set), the
set of Hopf differentials is itself very non-generic.
It is therefore necessary to study in detail the range of the map
\(I:\cU\to\mathcal{H}\) associating to each map $U\in \cU(\Omega)$ its Hopf differential \(I(U):=U_z^2\).
The heuristic idea is to consider an appropriate primitive
\[
U = 2 \int f^{1/2}.
\]
This procedure may end up with functions which are not in $\cU(\Omega)$. Actually, any $f$ belonging to the image of the Hopf differential $I$
is characterized by a system of equations (cf. \eqref{e.parti reali}
in Proposition \ref{p:from f to U}) and, hence, the residuality needs
to be proven for holomorphic functions solving suitable differential
constraints. The example discussed in \S \ref{ex.no U} shows
the high degree of rigidity of such conditions, suggesting that there are only finitely
many directions for perturbing a Hopf differential with higher order zeros into one with simple roots.
However, a careful and detailed analysis shows that this can be achieved through an iterative argument.

We think that
our approach is prototypical of a variety of other contexts where generically singularities with reduced complexity (such as triple junctions in the plane) are expected.

\begin{figure}[h!]
\begin{center}
\begin{tikzpicture}
\draw[thick] (0,0) circle (2cm);
  \coordinate (A) at (0.618, 1.90211);
    \coordinate (B) at (-1.61803, 1.17557);
     \coordinate (C) at (-1.61803, -1.17557);
    \coordinate (D) at (0.618, -1.90211);
     \coordinate (E) at (2,0);
\draw (0,0) -- (A);
\draw (0,0) -- (B);
\draw (0,0) -- (C);
\draw (0,0) -- (D);
\draw (0,0) -- (E);
\draw (A) circle (3pt);
\draw (B) circle (3pt);
\draw (C) circle (3pt);
\draw (D) circle (3pt);
\draw (E) circle (3pt);
\filldraw[black] (0,0) circle (2pt);
\end{tikzpicture}
\qquad
\begin{tikzpicture}
\draw[thick] (0,0) circle (2cm);
  \coordinate (A) at (0.618, 1.90211);
    \coordinate (B) at (-1.61803, 1.17557);
     \coordinate (C) at (-1.61803, -1.17557);
    \coordinate (D) at (0.618, -1.90211);
     \coordinate (E) at (2,0);
      \coordinate (P1) at (0.6,0.5);
       \coordinate (P2) at (0.08,-0.08);
         \coordinate (P3) at (-0.7,-0.1);
\draw (P1) -- (A);
\draw (P3) -- (B);
\draw (P3) -- (C);
\draw (P2) -- (D);
\draw (P1) -- (E);
\draw (P1) -- (P2);
\draw (P2) -- (P3);
\draw (A) circle (3pt);
\draw (B) circle (3pt);
\draw (C) circle (3pt);
\draw (D) circle (3pt);
\draw (E) circle (3pt);
\filldraw[black] (P1) circle (2pt);
\filldraw[red] (P2) circle (2pt);
\filldraw[green] (P3) circle (2pt);
\end{tikzpicture}
\caption{Five species: configuration with one 5-point (on the left);
configuration with three 3-points (on the right).}
\label{figtype2}
\end{center}
\end{figure}
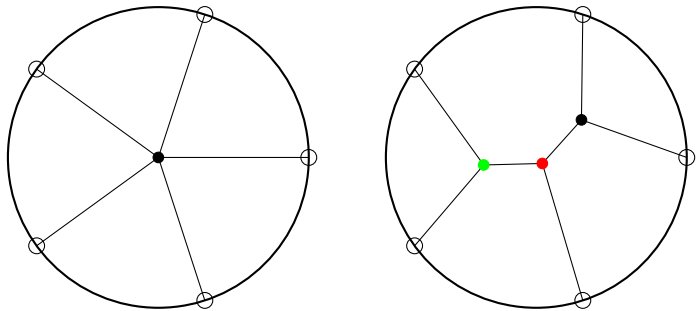

\subsection{Structure of the paper}
The paper is so organized. In Section \ref{s.2} after recalling some basic facts and
known results of the set \(\cU(\Omega)\), we give a characterization of the functions belonging to the Bergman space of integrable holomorphic functions which are Hopf differentials of stationary maps $U:\Omega \to \Sigma_N$ for any $N$.
Section \ref{s.3} is then entirely devoted to the proofs of the results
stated in Section \ref{s.2}. Next Section \ref{s:nodalset} contains
some results on the structure of the free boundary of the functions in \(\cU(\Omega)\)
and a generalization of the index formula proved in \cite{lm2} on the number of critical points counted with multiplicity. Section \ref{s:desingularization} deals with
the main analytical step of the genericity result, namely the desingularizing
procedure of a higher order zero point.
We need to take into account the linear
algebra constraints coming from the above hinted rigidity, in addition
to the global structure of holomorphic functions.
By reducing the order of the zeros one by one, after a finite
iteration, one can prove that Hopf differentials with simple zeros
are dense, thus providing the principal ingredient for the main
theorem (Figure \ref{figtype2} schematically shows the
desingularization of a $5$-point).

Finally, in Section \ref{s:generic} we prove
the main genericity result for harmonic maps \(\cU(\Omega)\) with points of multiplicity at most \(3\).

\section{Holomorphic functions and segregated states}\label{s.2}

We recall some basic facts about solutions to \eqref{e.dirichlet energy}.
For $N\geq 2$ the set of nontrivial stationary maps $U:\Omega\to \Sigma_N$ which are stationary for \eqref{e.dirichlet energy} is
characterized (see \cite{ctv2,ctv3,wz}) by the following analytical conditions:
\[
\cU_N(\Omega)=
\left\{
\begin{array}{c}
U=(u_1, \ldots, u_N) \hbox{ : } u_j\in H^1(\Omega),\\
\Delta u_j\geq 0 , \quad
\Delta \big(u_j-\sum_{k\neq j} u_k \big)\leq 0,\\
 u_j\geq0, \; u_j\not\equiv 0,\;
u_j\cdot u_k=0 \hbox{ a.e., } j\neq k \\
\end{array} \right\}.
\]
The set \(\cU_1(\Omega)\) is made of positive harmonic functions and
plays no role in our analysis.

\begin{remark}\label{r.indeterminato}
For $U\in \cU_N(\Omega)$ defined in \eqref{e.U}, it may happen that the nodal region $\omega_j =\{u_j>0\}$ (which will be shown to be an open set)
has an arbitrary number of connected components.
In the case the number of connected components of $\{U>0\}$ is finite, we assume that
\[
\omega_j =\{u_j>0\} \quad \text{is connected}
\quad \forall\; j=1, \ldots, N.
\]
Such condition is not restrictive, and can be achieved by a suitable relabeling.
\end{remark}
Let $\mathcal{H}(\Omega)$ be the set of the holomorphic functions and
\[
\mathcal{A}^1(\Omega):=\mathcal{H}(\Omega)\cap L^1(\Omega)
\]
its Bergman space. For \(f\in\mathcal{H}(\Omega)\) we
set \(\mathcal{Z}_f:=\{z\in \Omega:f(z)=0\}\).
We use the Wirtinger derivatives
\begin{align*}
& U_z:=\tfrac{1}{2} \left(  U_{x_1}
  -i U_{x_2} \right),\quad U_{\ov{z}}:=\tfrac{1}{2}
  \left(U_{x_1} +iU_{x_2} \right)\in L^2(\Omega;\ci).
\end{align*}



\subsection{The map \(I\)} A crucial role will be played by the Hopf differential.
The main result is contained in the following proposition.

\begin{proposition}\label{p.dez hol}
Let $\Omega\subset \ci$ be an open set. If \(U\in\cU(\Omega)\)  then \(U_z^2\in \mathcal{A}^1(\Omega)\).
\end{proposition}

We denote \(I:\mathcal{U}(\Omega)\to\mathcal{A}^1(\Omega)\) the map associating to each $U\in \mathcal U(\Omega)$ its Hopf differential \(I(U)=U_z^2\).
The proof is postponed to \S \ref{proof21}. A straightforward
consequence of Proposition \ref{p.dez hol} is that every  \(U\in \mathcal U(\Omega)\) is a locally Lipschitz
continuous in \(\Omega\), because its Wirtinger derivative is locally bounded (see 
Corollary \ref{c.reg1}).

The map \(I\) is not surjective, its range is actually a very small set of holomorphic 
functions. Characterizing \(I(\cU(\Omega))\) as a subset of \(\mathcal{A}^1(\Omega)\) is one of the 
main points of our analysis. To this aim we recall some elementary facts about
holomorphic functions.
Since the problem is invariant under conformal change of variables, if $\Omega$ is a bounded simply
connected domain, then we can assume, without loss of generality, that
\[
\Omega\equiv \D =\{z\in \ci : |z|<1\}.
\]
Therefore, we omit the dependence on the domain in \(\mathcal{U}_N\), \(\mathcal{U}\), \(\mathcal{H}\),
\(\mathcal{A}^1\).
The zero set of a non-constant holomorphic function 
\(f:\D\longrightarrow\ci\) is discrete. Therefore, the set
\(\mathcal{Z}^{\textup{odd}}_f\) of the zeros of \(f\) with odd order is at most 
countable. We write
\begin{equation}\label{e:ipotesi zeri}
\mathcal{Z}^{\textup{odd}}_f = \big\{z_1, z_2,\ldots \big\}.
\end{equation}
We fix a family of non-intersecting closed segments $L_j$ with an 
endpoint in \(z_j\) and the other on \(\pa\D\). Let
\(P\in\enne\cup\{+\infty\}\) be the number of such segments. We set
\begin{equation}\label{e:D tilde}
\Omega^- := \D\setminus \bigcup_{j=1}^P L_j.
\end{equation}
with the convention that if \(P=0\) then \(\Omega^-=\D\). 

\begin{figure}[h]
\begin{tikzpicture}
\draw[thick] (0,0) circle (20mm);
\draw (0.2,0.2) -- (0,2);
\draw (1,.5) -- (.55,1.91);
\draw (-.5,.5) -- (-1.628,1.168);
\draw (-.4,-.3) -- (-1.81,-.814);
\draw (.35,-1.25)--(0,-2);
\draw (0,0) circle (1pt);
\filldraw[black] (0.2,0.2) circle (1pt);
 \node at (0.45,.3) {$z_1$};
\filldraw[black] (1,.5) circle (1pt);
   \node at (1.25,.5) {$z_4$};
\filldraw[black] (-.5,.5) circle (1pt);
  \node at (-.5,.7) {$z_3$};
\filldraw[black] (-.4,-.3) circle (1pt);
  \node at (-.2,-.3) {$z_2$};
\filldraw[black] (.35,-1.25) circle (1pt);
  \node at (.6,-1.3) {$z_{5}$};
  \node at (1.7,1.8) {$\D$};
\end{tikzpicture}
\caption{The set \(\Omega^-\)}\label{Fcut}
\end{figure}
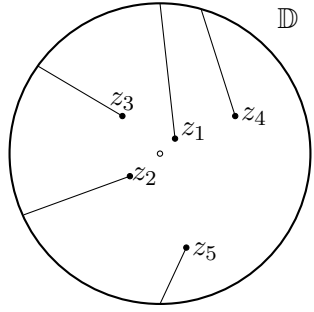

Since $f$ has no odd zeros in the simply connected domain $\Omega^-$, hence there exist exactly two holomorphic functions $\pm f^{1/2}:\Omega^-\longrightarrow\ci$
such that $(\pm f^{1/2})^2=f$.

We fix a point \(z_0\in\Omega^-\) and we set
\begin{equation}\label{prootnew}
F_{z_0,f}(z):=2\dyle\int_{z_0}^z f^{1/2}(\zeta)\, d\zeta.
\end{equation}

\begin{remark}
Both \(f^{1/2}\) and \(F_{z_0,f}\) are holomorphic in \(\Omega\)
and extend continuously to \(\Om^-\cup\mathcal{Z}_f^{\textup{odd}}\).
In particular, we may choose \(z_0\in\mathcal{Z}_f^{\textup{odd}}\) in \eqref{prootnew}.
\end{remark}

The characterization of the map \(I:\cU\longrightarrow\mathcal{A}^1\) 
is given in terms of \(F_{z_0,f}\). 

\begin{proposition}\label{p:from f to U}
For $f\in \mathcal{A}^1$ the following are equivalent:
\begin{itemize}
\item[i.] there exists $U\in\mathcal{U}$ such that $I(U)=f$;
\item[ii.] there exists $z_0\in \Omega^-\cup \mathcal{Z}_f^{odd}$ such 
that $F_{z_0,f}$ satisfies
\begin{equation}\label{e.parti reali}
\Re F_{z_0,f}(z)=0\quad \forall\;z\in \mathcal{Z}_f^{\textup{odd}}.
\end{equation}
\end{itemize}
Moreover, if \textup{ii.} holds, then $|\Re F_{z_0,f}|\in\mathcal{U}$
and $I(|\Re F_{z_0,f}|) = f$.
\end{proposition}

The proof is postponed to the \S \ref{proof23}.
If \(\cZ_f^{\textup{odd}}\) is empty, then ii. is always verified for every 
\(z_0\) and \(|\Re F_{z_0,f}|\in\cU\). In particular, the function \(U\) such 
that \(f=I(U)\) is not unique. Uniqueness, on the contrary, holds in the case 
\(\cZ_f^{\textup{odd}}\neq\emptyset\).

\begin{proposition}\label{p:unique}
Let \(f\in I(\cU)\).
\begin{itemize}
\item[i.] If $\mathcal{Z}_{f}^{\textup{odd}} \neq \emptyset$, 
then there exists a unique function $U\in\cU$ such that $I(U)=f$.
\item[ii.] If $\mathcal{Z}_{f}^{\textup{odd}}=\emptyset$, 
then
\(\big\{|\Re F_{z_0,f}| : z_0\in \D\big\} = I^{-1}(f)\).
\end{itemize}
\end{proposition}

The proof of the proposition is postponed to the \S \ref{proof25}.

\begin{remark}
Propositions \ref{p:from f to U} and \ref{p:unique} imply the commutativity 
of the following diagram:
\begin{equation}\label{diag}
\begin{picture}(150,33)(0,0)
\put(0,20){\(f_U\in\mathcal{A}^1\) }
\put(70,26){}
\put(45,23){\vector(1,0){65}}
\put(112,20){ \(F_U\in \mathcal{H}(\Omega^-)\) }
\put(90,-6){\(\pi\)}
\put(15,-10){\vector(0,1){25}}
\put(0,-20){\(U\in\mathcal{U} \)}
\put(3,-2){\(I\)}
\put(112,15){\vector(-3,-1){78}}
\end{picture}
\end{equation}
\vspace{3mm}

\noindent with \(f_U=I(U)\), \(F_U:=F_{z_0,f_U}\) with $z_0\in \Omega^-\cup \mathcal Z_f^{\textup{odd}}$ such that \(U(z_0)=0\) and
$\pi:\mathcal{H}(\Omega^-)\to \cU$ the nonlinear 
projection
\[
\pi(F)(z) := \left|\Re(F(z))\right|.
\]
\end{remark}

\section{Proofs of Propositions \ref{p.dez hol}, \ref{p:from f to U} and \ref{p:unique}}\label{s.3}

\subsection{Proof of Proposition \ref{p.dez hol}}\label{proof21}
The result is a direct consequence of the fact that
$\mathcal{U}_N$ are the stationary points of the Dirichlet energy ${E}$ in \eqref{e.dirichlet energy} (see \cite{wz}).
Fix any $U\in \mathcal{U}$ and compute inner variations of ${E}$ at $U$: namely, for every smooth vector field $\lambda:\Omega\to\ci$ with compact support in $\Omega$,
we consider the diffeomorphism (for $|\eps|$ sufficiently small)
\[
\Phi_\eps:\Omega\to\Omega, \qquad \Phi_\eps(z):=z+\eps\lambda(z),
\]
and set $U^\eps := U\circ \Phi_\eps$.
By construction $\textup{supp}(U^\eps- U)\subset\subset\Omega$
for every $\eps$ sufficiently small and by stationarity
\begin{equation}\label{e:inner}
\frac{\d}{\d\eps}\Big\vert_{\eps=0} {E}(U^\eps)
=  0\qquad \forall\lambda\in C_c^1(\Omega).
\end{equation}
The fact that $f_U$ is holomorphic is a direct consequence 
of \eqref{e:inner}.
In order to compute the inner variations, we write the 
Dirichlet energy in terms of the Wirtinger operators:
\[
{E}(U^\eps) = \frac{i}{2}\int_{\Omega} \Big(|U^\eps_z|^2
+ |U^\eps_{\bar z}|^2\Big) \, \d z\wedge\d \bar z,
\]
Using the chain rule for $V\in H^{1}(\ci,\erre)$ and
\(\Phi:\ci\to \ci\) smooth
\begin{align*}
\big(V\circ \Phi\big)_z &:= (V_z \circ \Phi)\, \Phi_z 
  + (V_{\bar z} \circ \Phi)\, \bar \Phi_{z},\\
\big(V\circ \Phi\big)_{\bar z} &:= (V_z \circ \Phi)\, \Phi_{\bar z} 
  + (V_{\bar z} \circ \Phi)\, \bar \Phi_{\bar z},
\end{align*}
and the formulas $\overline{\Phi_z}=\bar \Phi_{\bar z}$ and 
$\overline{\Phi_{\bar z}}=\bar \Phi_{z}$, we can compute that 
\begin{align*}
 |\big(V\circ \Phi\big)_z|^2 
&:= |V_z \circ \Phi|^2\, \big(\Phi_z \overline{\Phi_z}
  +\Phi_{\bar z}\overline{\Phi_{\bar z}}\big)
    +(V_{z} \circ \Phi)^2 \;\Phi_z \Phi_{\bar z} \\
& \qquad + (\overline{V_{z} \circ \Phi})^2\; \overline{\Phi_z}\;
  \overline{\Phi_{\bar z}}, \\
 |\big(V\circ \Phi\big)_{\bar z}|^2 
&:= |V_{\bar z} \circ \Phi|^2\, \big(\Phi_z \overline{\Phi_z}
  +\Phi_{\bar z}\overline{\Phi_{\bar z}}\big) 
  + (\overline{V_{\bar z} \circ \Phi})^2 \;\Phi_z \Phi_{\bar z} \\
& \qquad + (V_{\bar z} \circ \Phi)^2\; \overline{\Phi_z}\;
  \overline{\Phi_{\bar z}},
\end{align*}
and summing up the two identities
\begin{align*}
  |\big(V\circ \Phi\big)_z|^2 +|\big(V\circ \Phi\big)_{\bar z}|^2 
&:= \Big[|V_z \circ \Phi|^2 + |V_{\bar z} \circ \Phi|^2\Big]\, 
  \big(\Phi_z \overline{\Phi_z}+\Phi_{\bar z}\overline{\Phi_{\bar z}}\big) \\
&\quad + 2 (V_{z} \circ \Phi)^2 \;\Phi_z \Phi_{\bar z}
  +2(V_{\bar z} \circ \Phi)^2\; \overline{\Phi_z}\;\overline{\Phi_{\bar z}}.
\end{align*}
We use the previous identity with $V=U$ and 
$\Phi(z)=\Phi_\eps(z)=z+\eps\lambda(z)$,
\[
\Phi_z =1+\eps\lambda_z, \qquad \Phi_{\bar z} =\eps\lambda_{\bar z},
\]
which implies 
\begin{align*}
  |U^\eps_z|^2 +|U^\eps_{\bar z}|^2 
&:= \Big[|U_z \circ \Phi_\eps|^2 + |U_{\bar z} \circ \Phi_\eps|^2\Big]\, 
  \big(1+\eps (\lambda_z+ \overline{\lambda_z}) + O(\eps^2)\big) \\
&\quad +2\eps \,(U_{z} \circ \Phi_\eps)^2 \;\big(\lambda_{\bar z}
  + O(\eps^2)\big) +2\eps\,(U_{\bar z} \circ \Phi)^2\; 
  \big(\overline{\lambda_{\bar z}}+ O(\eps^2)\big).
\end{align*}
We can then write
\begin{align*}
  E(U^\eps)
& = \frac{i}{2}\int_{\Omega} \Big(|U^\eps_z|^2
  + |U^\eps_{\bar z}|^2\Big) \, \d z\wedge\d \bar z \\
& = \frac{i}{2}\int_{\Omega}\Big[|U_z \circ \Phi_\eps|^2
  + |U_{\bar z} \circ \Phi_\eps|^2\Big]\, 
  \big(1+\eps (\lambda_z+ \overline{\lambda_z}) 
   +O(\eps^2)\big)\, \d z\wedge\d \bar z \\
&\quad +\eps i\int_{\Omega}\,(U_{z} \circ \Phi_\eps)^2 \;
  \big(\lambda_{\bar z}+ O(\eps^2)\big)\;\d z\wedge\d \bar z \\
&\qquad +\eps i\int_{\Omega}(U_{\bar z} \circ \Phi_\eps)^2\;
  \big(\overline{\lambda_{\bar z}}+ O(\eps^2)\big)\d z\wedge\d \bar z.
\end{align*}
We now make the change of variables $z= \Phi_\eps^{-1}(\zeta)$, 
which in turn implies
\begin{align*}
\d z \wedge \d \bar z 
  & = \left[(\Phi_\eps^{-1})_\zeta \, \d \zeta 
    +(\Phi_\eps^{-1})_{\bar \zeta}\, \d \bar\zeta\right]\wedge 
    \left[(\overline{\Phi_\eps^{-1}})_\zeta \, \d \zeta + 
    (\overline{\Phi_\eps^{-1}})_{\bar \zeta}\, \d \bar\zeta\right] \\
  & = \left[(\Phi_\eps^{-1})_\zeta \, (\overline{\Phi_\eps^{-1}})_{\bar \zeta}  
    -(\Phi_\eps^{-1})_{\bar \zeta}\,(\overline{\Phi_\eps^{-1}})_\zeta \right] 
    \d \zeta\wedge\d \bar\zeta.
\end{align*}
In particular, since 
$\Phi_\eps^{-1}(\zeta)=\zeta -\eps\lambda(\zeta) +O(\eps^2)$, we get 
\begin{align*}
  \d z \wedge \d \bar z 
& =  (1-\eps\lambda_{\zeta}+O(\eps^2))\,(1-\eps \bar\lambda_{\bar \zeta}
  +O(\eps^2))\;\d \zeta\wedge\d \bar\zeta \\
& \qquad-(-\eps\lambda_{\bar\zeta}+O(\eps^2))\,
  (-\eps\bar\lambda_{\zeta}
  +O(\eps^2))\;\d \zeta\wedge\d \bar\zeta \\
& =\big[ 1-\eps\big(\lambda_{\zeta}+\bar\lambda_{\bar \zeta}\big)
  +O(\eps^2)\big]\d \zeta\wedge\d \bar\zeta.
\end{align*}
We can then rewrite the Dirichlet energy with respect to the
variable $\zeta$:
\begin{align*}
  E
& (U^\eps) = \frac{i}{2}\int_{\Omega}\Big[|U_\zeta|^2 + |U_{\bar\zeta}|^2\Big]\,
  \big(1+\eps (\lambda_\zeta\circ \Phi_\eps^{-1} 
  +\overline{\lambda_\zeta}\circ \Phi_\eps^{-1}) + O(\eps^2)\big)\, \cdot \\
& \qquad\cdot\big[ 1-\eps\big(\lambda_{\zeta}
  +\bar\lambda_{\bar \zeta}\big)
  +O(\eps^2)\big] \d \zeta\wedge\d \bar \zeta \\
& \quad+  \eps i \int_{\Omega}\,(U_{\zeta})^2 \;
  \big(\lambda_{\bar \zeta} \circ \Phi_\eps^{-1}+ O(\eps^2)\big)\;
  \big[ 1-\eps\big(\lambda_{\zeta}
  +\bar\lambda_{\bar\zeta}\big)+O(\eps^2)\big] 
  \d \zeta\wedge\d \bar\zeta \\
& \quad +\eps i\int_{\Omega}\,(U_{\bar\zeta})^2\;
  \big(\overline{\lambda_{\bar \zeta}}\circ \Phi_\eps^{-1}
  + O(\eps^2)\big)\; \big[ 1-\eps\big(\lambda_{\zeta}
  +\bar\lambda_{\bar \zeta}\big)+O(\eps^2) \big] \d \zeta\wedge\d 
  \bar \zeta.
\end{align*}
Taking the derivative in $\eps=0$ of the energy $E(U^\eps)$
is now straightforward to infer that 
\begin{align}\label{e.inner var}
  \frac{\d}{\d\eps}\Big\vert_{\eps=0} E(U^\eps)
& = \frac{i}{2}\int_{\Omega}\Big[|U_\zeta|^2 + |U_{\bar \zeta}|^2\Big]\,
  (\lambda_\zeta +\overline{\lambda_\zeta})\,\d \zeta\wedge\d 
  \bar\zeta \notag \\
& \quad -\frac{i}{2}\int_{\Omega}\Big[|U_\zeta|^2 +|U_{\bar \zeta}|^2\Big]\,
  (\lambda_\zeta + \overline{\lambda_\zeta})\,
  \d \zeta\wedge\d \bar\zeta \notag \\
& \quad +i\int_{\Omega}\,U_{\zeta}^2 \;\lambda_{\bar\zeta}\;
  \d \zeta\wedge\d \bar\zeta
  +i\int_{\Omega}\,U_{\bar \zeta}^2\;\overline{\lambda_{\bar \zeta}}\;
  \d \zeta\wedge\d \bar\zeta \notag \\
& =2i \int_{\Omega}  \Re\big(U_{\zeta}^2 \;\lambda_{\bar \zeta}\big)\;
    \d \zeta\wedge\d \bar \zeta = 0 \qquad 
    \forall\lambda\in C_c^\infty(\Omega;\ci).
\end{align}
Recalling that $\lambda$ takes values in $\ci$, by its arbitrariness we conclude
that $(U_z^2)_{\bar z}=0$, i.e. $U_z^2$ satisfies the Cauchy-Riemann 
equations in a weak sense and therefore is holomorphic in $\Omega$, since $U_z^2\in L^1(\Omega)$.
\hfill\qed

\bigskip

A simple corollary is the Lipschitz continuity of the functions in $\mathcal U$.

\begin{corollary}
\label{c.reg1}
Let \(U\in \cU\), with \(U=u_1+\cdots + u_N\in\cU_N\) for some 
$N\geq 2$. Then, the functions $U$ and $u_j$ are locally Lipschitz 
continuous and the open sets $\omega_j :=\{u_j>0\}\subset \Omega$  have no connected component compactly
contained in $\Omega$.
\end{corollary}

\begin{proof}
The Lipschitz continuity of $U$ follows from the fact that 
$(U_z)^2=f_U$ is holomorphic and, hence, locally bounded in $\Omega$. 
As a consequence, also the functions $u_j$ are locally Lipschitz 
(because they are the zero extension of the Lipschitz functions 
$U\vert_{\omega_j}$, which in turns satisfy 
$U\vert_{\pa\omega_j\cap \Omega}=0$). Moreover, by the very definition 
of $\cU$ it follows that $u_j$ are harmonic in $\omega_j$. Therefore, there cannot exist connected components $C\subset\omega_j$ with
$\overline C\cap\partial \Omega=\emptyset$, because otherwise 
$u_j\vert_C$ would be a positive harmonic function with $u_j=0$ 
on $\pa C$.
\end{proof}

\subsection{Definition and properties of \(F_{z_0,f}\)}
Given any holomorphic function $f \in \mathcal{H}$, we use
the notation introduced in \S \ref{s.2}.

%
The main step of the proof of Propositions \ref{p:from f to U} 
and \ref{p:unique} is contained in the following lemma.

\begin{lemma}\label{l:from f to U}
Let $f\in \mathcal{A}^1$ and \(F_{z_0,f}\) defined in \eqref{prootnew} with \(z_0\in \Omega^-\cup \cZ_f^{odd}\).
If $\Re F_{z_0,f}(z)=0$ for all $z\in \mathcal{Z}_f^{\textup{odd}}$, 
then $|\Re F_{z_0,f}|$ extends by continuity to a function $U\in \overline\cU$.
\end{lemma}

\begin{remark}
In the case $\mathcal{Z}_f^{\textup{odd}}=\emptyset$, the hypotheses 
are automatically satisfied and the lemma asserts that 
$|\Re F_{z_0,f}|\in \mathcal{U}$ (see also \cite[Proposition 3.9]{lm2}, where the case of solutions with singular points with even multiplicity is explicitly considered).
\end{remark}

\begin{proof}
Assume $\mathcal{Z}_f^{\textup{odd}}\neq\emptyset$. Fix an index $j$ and any point on a cut $\xi\in L_j\cap \D$,
$\xi\neq z_j$. We consider the limits of $F_{z_0,f}$ as $z$ tends
to $\xi$ from the two sides of the cut $L_j$, i.e.
\[
z\to \xi^\pm \quad \Longleftrightarrow \quad
\begin{cases}
z\to \xi,\\
\pm\left(\textup{Arg}(z-z_j) - \textup{Arg}(\xi-z_j)\right) >0 ,
\end{cases}
\]
for any local determination of the argument of $(\xi-z_j)$ (see Figure \ref{dete_argument}).
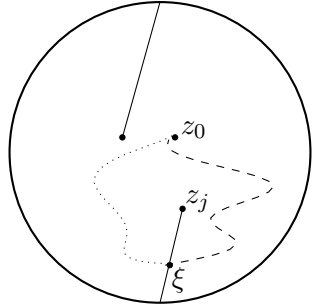
\begin{figure}[h]
\begin{tikzpicture}
\draw[thick] (0,0) circle (20mm);
\draw (-0.5,0.2) -- (0,2);
\draw  (.3,-.75)--(0,-2);
\filldraw[black] (-0.5,0.2) circle (1pt);
\filldraw[black] (0.2,0.2) circle (1pt);
 \node at (0.45,.3) {\(z_0\)};
\filldraw[black] (.3,-.75) circle (1pt);
  \node at (.5,-.65) {$z_j$};
\filldraw[black] (.13,-1.5) circle (1pt);
  \node at (.25,-1.7) {$\xi$};
\draw [domain=0:1.,dotted]
  plot({.16-.89*sin(pi*\x r)+.2*sin(4*pi*\x r)},{-.63+.85*cos(pi*\x r)});
\draw [domain=0:1,dashed]
  plot({.07+.17*\x+4*\x*(1-\x)+.4*sin(3.9*pi*\x r)},{-1.5+1.67*\x});
\end{tikzpicture}
\caption{\(z\to\xi^-\) along the dotted path, and
\(z\to\xi^+\) along the dashed path}\label{dete_argument}
\end{figure}

Namely, we consider
\begin{align}\label{e:tilde F}
F_{z_0,f} (\xi^\pm) :&= 2\lim_{z\to \xi^\pm}\int_{z_0}^z
  f^{1/2}(\zeta) \, d\zeta\notag\\
&= F_{z_0,f} (z_j) + 2\lim_{z\to \xi^\pm}\int_{z_j}^z f^{1/2}(\zeta) d\zeta.
\end{align}
Observe that $f^{1/2}$ changes sign across $L_j$, i.e.
\[
\lim_{z\to \xi^+} f^{1/2}(z) = -\lim_{z\to \xi^-} f^{1/2}(z),
\]
because $z_j$ is assumed to be a zero with odd order.
Therefore, \eqref{e:tilde F} reads as
\[
F_{z_0,f} (\xi^+) -F_{z_0,f} (z_j) =-F_{z_0,f} (\xi^-) +F_{z_0,f} (z_j).
\]
Using the hypothesis $\Re F_{z_0,f}(z_j) = 0$,
we infer that
\[
\Re F_{z_0,f} (\xi^+) = -\Re F_{z_0,f}(\xi^-),
\]
and, hence, the following limit is well-defined
\[
|\Re F_{z_0,f}(\xi)|:=\lim_{z\to \xi} |\Re F_{z_0,f}(z)|.
\]
We then conclude that there exists a continuous extension 
$V$ of $|\Re F_{z_0,f}|$ to the whole $\D$.

The extension $V$ is indipendent from the choice of the cuts $L_j$.
Indeed, consider a different simply connected domain $\widetilde \Omega$ obtained by taking away a set of non-intersecting segments
$\widetilde L_j$ with $L_j
=\widetilde L_j$ for every $j$ except one index $j_0$ (see Figure \ref{f.bo}).
\begin{figure}[h]
\begin{tikzpicture}
\draw[thick] (0,0) circle (20mm);
\draw (0.2,0.2) -- (0,2);
\filldraw[black] (0.2,0.2) circle (1pt);
\draw  (.3,-.75)--(0,-2);
\filldraw[black] (.3,-.75) circle (1pt);
 \node at (0,-.7) {$z_{j_0}$};
 \node at (-.1,-1.2) {$L_{j_0}$};
\draw  (.3,-.75)--(2,0);
  \node at (1.5,.1) {$\tilde{L}_{j_0}$};
\end{tikzpicture}
 \caption{The domains $\Omega^-$ and $\widetilde \Omega$.}
 \label{f.bo}
 \end{figure}
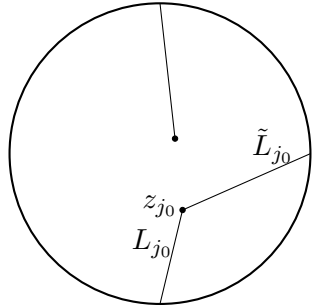
Then, we can choose the determinations
of the square root in such a way that $\widetilde f^{1/2}=f^{1/2}$
in one of the two connected components of $\Omega^-\cap \widetilde\Omega$,
and $\widetilde f^{1/2}=-f^{1/2}$ in the other.
In computing $F_{z_0,f}(z)$, we can consider paths $\gamma:[0,1]\to \D$ with $\gamma(0) = z_0$, $\gamma(1/2)= z_{j_0}$ and $\gamma(1) = z$, with $\gamma((0,1/2))$ and $\gamma((1/2,1))$ contained inside a connected
component of $\Omega^-\cap \widetilde\Omega$. Therefore, it follows from \eqref{e:tilde F} that $\Re F_{z_0,f}=\pm\Re\widetilde F_{z_0,f}$, 
thus concluding that
\[
|\Re F_{z_0,f}| = |\Re \widetilde F_{z_0,f}|,
\]
from which the uniqueness of the extension $V$ follows.

We denote by $\{\omega_j\}_j$ the connected
components of $\{V>0\}$ and set
\[
v_j:=
\begin{cases}
V(x) & \textup{if } x\in\omega_j,\\
0 & \textup{elsewhere}.
\end{cases}
\]
Note that $v_j$ is subharmonic, because $v_j$ is a continuous nonnegative function which is harmonic where it is positive.
We remark that $V$ is locally the modulus of a harmonic function in $\D\setminus \cZ_f^{\textup{odd}}$; therefore, the level set $\{V=0\}$ is made by analytic curves with endpoints on $\cZ_f\cup\partial\D$.
This implies that for every $j_0$ and for every  $z\in \partial \omega_{j_0}\cap \D \setminus \cZ_f $ there exists exactly one $i_0$ such that $z\in \partial \omega_{j_0} \cap \partial \omega_{i_0}$.
Therefore, 
\[
\tilde V := v_{j_0}-\sum_{j\neq {j_0}}v_j 
\]
equals the harmonic function $v_{j_0} - v_{i_0}=\pm \Re F_{z_0,f}$ in a neighborhood of $z$.
Moreover, $\tilde V$ is superharmonic away from $\overline{\omega_{j_0}}$, because it is the sum of superharmonic functions $-v_j$, $j\neq j_0$.
This implies that globally $\tilde V$ is superharmonic in $\D\setminus (\cZ_{f}^{}\cap \{V=0\})$.

To complete the proof, we need to verify that the Euler-Lagrange equations 
are satisfied also in a neighborhood of the isolated points $ \cZ_f^{}\cap \{V=0\}$. Consider any test function $\phi\in C_c(\D)$ with $\textup{supp} (\phi)\cap \cZ_f^{}\cap \{V=0\} = \{z_{j}\}$ and consider a smooth cut-off function $\chi_s$ such that
$\chi_s(z) = 1$ if $|z-z_j|>2s$ and $\chi_s(z) = 0$ if $|z-z_j|<s$, with 
\[
s\,\|\partial_z \chi_s\|_\infty +s^2\,\|\partial^2_{z\bar z} \chi_s\|_\infty\leq C.
\]
A simple computation yields
\begin{align*}
\int_{\D} v_{j_0} \, \partial^2_{z\bar z} \phi\, & = \lim_{s\to 0^+}\int_{\D} v_{j_0} \, \partial^2_{z\bar z} (\chi_s\,\phi)\quad \forall \; {j_0} \ \textup{with } z_{j_0}\in \cZ_f^{}\cap \{V=0\} ,
\end{align*}
taking into account that $v_{j_0}(z_j)= 0$.
This implies that $v_{j_0}$ if subharmonic in the whole of $\D$, as well as $v_{j_0}-\sum_{j\neq {j_0}} v_j$ superharmonic in $\D$.

Finally, we verify that $V$ belongs to $\overline\cU$.
To this aim, we notice that for every $r<1$ there exists an integer $N_r$ such that $V\vert_{\D_r} \in \cU_{N_r}(\D_r)$ with $\|V\|_{H^1(\D_r)}\leq \|f\|_{L^1(\D)}$. By a simple rescaling we deduce that $V(r\cdot)\in \cU$ and $V(r\cdot)\to V$ in $H^1(\D)$, thus $V\in \overline\cU$.

\end{proof}

\subsection{Commutativity of the diagram \eqref{diag}}
Next we show that for every $U\in \mathcal{U}$ the function $f_U := I(U)$ satisfies \eqref{e.parti reali}  of Proposition \ref{p:from f to U}.

\medskip

We need the following lemma.


\begin{lemma}\label{l.parte reale} 
Let \(h(z)\) be holomorphic, $u = \Re h$ e $v= \Im h$.
Then, $4\big(\partial_z|u|\big)^2=(h')^2$ for a.e. $z$.
\end{lemma}

\begin{proof}
We have,
\[
h'(z)= h_{x_1}=u_{x_1} + i v_{x_1} = u_{x_1}-iu_{x_2}= 2u_z,
\]
where we used the Cauchy-Riemann conditions
\( u_{x_1}=v_{x_2} \), \(u_{x_2}=-v_{x_1}\).
Next, for every $z$ such that $u(z)\neq 0$ (recall that $\{u =0\}$
is a negligible set because of the harmonicity of $u$), we can compute as follows
\[
\begin{aligned}
4 |u|_z^2
&
  = \Big[ \frac{u}{|u|}\Big(\frac{\partial }{\partial x_1}
  -i\frac{\partial }{\partial x_2}\Big) u \Big]^2
  =4u_z^2=(h'(z))^2.
\end{aligned}
\]
\end{proof}

We can then show the following.

\begin{proposition}\label{p:from U to U}
Let $U\in\mathcal{U}$, $f_U:=I(U)$ and $F_{z_0,f_U}$ be as in \eqref{prootnew} with $z_0\in \{U=0\}$ any fixed point. Then, $\Re F_{z_0,f_U} (z) = 0$ for all $z\in \mathcal{Z}_{f_U}^{\textup{odd}}$ and
$U=|\Re F_{z_0,f_U}|$.
\end{proposition}

\begin{proof}
Let $\Omega^-\subseteq \D$ be the domain in \eqref{e:D tilde} for the holomorphic function $f_U$.
We set $V(z):= |\Re F_{z_0,f_U}(z)|$. Then, by Lemma \ref{l.parte reale} we have that
\begin{equation}\label{e.derivate pm}
V_z^2 = f_U = U_z^2\quad \textup{and} \quad V(z_0)=U(z_0)=0.
\end{equation}
We claim that \eqref{e.derivate pm} implies that $U=V$.
Condition \eqref{e.derivate pm} implies that $U_z = \pm V_z$ a.e. By the regularity of $U$, the sign is constant in any connected component of $\{U>0\}\cap \{V>0\}$.
If $U_z = - V_z$, then $U+V=C>0$ with $C\in \erre$. If $U_z = V_z$, then $U-V=C$ with $C\in \erre$.
In particular, for any connected component $A\subseteq\{U>0\}\cap \{V>0\}$ it holds that
\[
\partial A \subseteq \{U=0, \, V=C\} \cup \{U=C, \; V=0\}.
\]
Since $\{V=C\}$ are locally the level sets of a harmonic function, there cannot be accumulation of the level sets $\{V=C\}\cap \{U=0\}$ towards $z_0$. This implies that $z_0$ is on the boundary of some connected components of $\{U>0\}\cap \{V>0\}$, and since $U(z_0)= V(z_0)=0$ actually $U=V$ in a neighborhood of $z_0$. By unique continuation, it then follows that $U=V$ is an open subset and, by connectedness, in the whole of $\Omega^-$. Therefore, we conclude that $V$ can be extended by continuity in $\D$ as $U=V$.
To conclude the proof we need only to notice that
\[
\mathcal{Z}_{f_U}^{\textup{odd}} \subset \{U=0\},
\]
because in a neighborhood of any point of $\mathcal{Z}_{f_U}\cap \{U>0\}$ the function $U$ is harmonic and $U_z^2$ in such neighborhood can only have zeros of even order. This implies that $\Re F_{z_0,f_U}(z) = 0$ for all $z\in \mathcal{Z}_{f_U}^{\textup{odd}}$.
\end{proof}

\subsection{Proof of Proposition \ref{p:from f to U}}\label{proof23}
We start showing that i. implies ii. in Proposition \ref{p:from f to U}, i.e., we show that if $f=I(U)$ for some $U\in \mathcal U$, then there exists $z_0$ such that
\[
\Re F_{z_0,f}(z)=0\quad \forall\;z\in \mathcal{Z}_f^{\textup{odd}} \quad \textup{and}\quad |\Re F_{z_0,f}|\in\mathcal{U}.
\]
In fact, this is a direct consequence of Proposition
\ref{p:from U to U}, if $z_0$ is any point such that $U(z_0)=0$.

Viceversa, ii. implies i.: indeed, by Lemma
\ref{l:from f to U}, if
\[
\Re F_{z_0,f}(z)=0\quad \forall\;z\in \mathcal{Z}_f^{\textup{odd}},
\]
then $|\Re F_{z_0,f}|\in\cU$ and by Lemma \ref{l.parte reale} $I(|\Re F_{z_0,f}|)=f$.
\hfill\qed

\subsection{Proof of Proposition \ref{p:unique}}\label{proof25}
i. Assume that $f=I(U)$ with $U\in \cU$. Then, $\emptyset\neq\mathcal{Z}_{f}^{\textup{odd}}\subset \{U=0\}$ implies that there exists $z_0\in \mathcal{Z}_f^{\textup{odd}}$ such that $U=|\Re F_{z_0,f}|$ by Proposition \ref{p:from U to U}. In particular, $U$ depends uniquely on $f$ and it is the only function with $I(U)=f$.

For ii. we notice that by Lemma \ref{l.parte reale} we have that any function of the form $V=|\Re F_{z_0,f}|$ with $z_0\in \D$ is defined in the whole $\D$, belongs to $\mathcal{U}$ and satisfies $I(V) = f$.
Viceversa, any $V$ such that $I(V)= f$ can be
recovered by Proposition \ref{p:from U to U} as $V= |\Re F_{z_0,f}|$, where $z_0$ is any point in $\{V=0\}$, thus showing that
\[
\big\{|\Re F_{z_0,f}| : z_0\in \D\big\} = I^{-1}(f).
\]

\hfill\qed

%

\section{Nodal set of the functions in $\cU$}\label{s:nodalset}

The main result of this section (Proposition \ref{pr:formula}) is the 
proof of a Euler-type formula relating the number of vertices, edges 
and faces of a planar graph (see \cite[Theorem 1.5.2]{j}), applied to the 
study of the nodal set of \(U\in\cU\).  
In order to do so, we need to show that the nodal
set of \(U\) can be viewed as a graph with some very special properties. 

Therefore we need to introduce the definition of nodal and critical 
set of \(U\in \cU\) and to relate to the order of the zeros of \(f_U\). 
This is done in Proposition \ref{p.reg2}. 

We start by reviewing results which are already known 
in the literature \cite{ctv0,ctv3}.

\subsection{Structure of $U\in \cU$}
We give direct proofs of several regularity properties of the 
functions $U\in \cU$.
Recall the definition of $m_U(z)$ in \eqref{e.multiplicity m}.
If $m_U(z_0)=1$, then in a neighborhood of $z_0$ there exists only one
function, say $u_{j_0}$, not identically zero. By Harnack inequality,
$u_{j_0}>0$ in that neighborhood and $z_0\in \omega_{j_0}\subset\{U>0\}$. 
Therefore, the set of positivity of $U$ is equivalently characterized 
as the set of multiplicity one points:
\[
\{U>0\} = \big\{z\in\D : m_U(z)=1\big\}.
\]
Moreover, there are no points with $m_U(z_0)=0$, because otherwise
$U\equiv 0$ in a neighborhood of $z_0$ with the consequence that 
both $\nabla U$ and $f_U$ are identically zero.
The $2$-points characterize the interfaces
between the regions $\omega_j$. We denote $\cN_U,\mathcal{C}_U$ 
the nodal and critical set of $U$, respectively:
\[
\cN_U:=\{z\in\D: U(z)=0\}\quad \textup{and}\quad 
\mathcal{C}_U:=\{z\in \D : U(z) = f_U(z) = 0\}.
\]
The next result links \(m_U(z_0)\) with the order of the zero \(z_0\) 
of \(f_U\).

\begin{proposition}\label{p.reg2}
Let \(U\in\cU\). Then,
\begin{equation}\label{formula-x}
m_U(z_0)=2+{\rm ord}(\fU;z_0)\qquad \forall\; z_0\in \cN_{U},
\end{equation}
and the following holds:\\
i. $\mathcal{C}_U=\{z\in\D: m_U(z)\geq 3\}$.\\
ii. $\cN_U$ is the union of non-intersecting analytic
arcs with endpoints on $\mathcal{C}_{U}\cup \partial \D$, 
which meets in $\mathcal{C}_U$ at equal angles.\\
iii. If $U=u_1\cdots u_N$, then
$u_\ell-u_j$ is harmonic in $\textup{Int} (\overline \omega_\ell \cup \overline \omega_j)$.\\
iv. For every \(z_0\in\D\) such that \(m_U(z_0)=2\) there exist \(\ell\neq j\)
such that \(z_0\in \overline \omega_\ell \cap \overline \omega_j\) and
\begin{equation}\label{grad}
\lim_{\stackrel{z\to z_0}{z\in \omega_\ell}} \nabla u_\ell(z)
=-\lim_{\stackrel{z\to z_0}{z\in \omega_j}} \nabla u_j(z) \neq 0.
\end{equation}
v. If $z_0\in \D$ with $m_U(z_0)=h\geq 3$, then there exists
\(\th_0\in (-\pi,\pi]\) such that
\begin{equation}\label{local}
U(z_0+r e^{i\th})=r^{h/2} \left| \cos\big( \tfrac{h}{2} (\th +\th_0)\big)
\right|+o(r^{h/2})\quad \textup{as $r\to 0$.}
\end{equation}
In particular, $\|\nabla U\|_{L^\infty(\D_r(z_0))} \to 0$ as $r\to 0$.
\end{proposition}

%
%

\begin{proof}
Let $z_0\in\cN_U$. If $f_U(z_0)\neq 0$, then we know that 
$U=|\Re F_{z_0, f_U}|$ is the absolute value of a harmonic function 
in a neighborhood of $z_0$ and therefore $\{U=0\}$ is locally an 
analytic arc separating two connected components of $\{U>0\}$. 
In this case, \eqref{formula-x} is verified because
\[
m_U(z_0) = 2 \qquad \hbox{and}\qquad {\rm ord}(f_U;z_0)=0.
\]
Viceversa let \(z_0\in \cN_U\cap {\cZ}_\fU\) be a zero of $f_U$ of order
\(n\geq 1\). Without loss of
generality, assume that \(z_0=0\) and in a neighborhood we can write
\[
\fU(z)=z^{n} (a_n+a_{n+1}z+...),\quad a_n\neq 0.
\]
If $n$ is odd, in a sufficiently small neighborhood of $z_0=0$ we can
select a single valued branch of the square root of the term inside
the parentheses, say
\[
(a_n+a_{n+1}z+...)^{1/2}=b_0+b_1 z+b_2 z^2+\cdots
\]
 We then get for \(f_U^{1/2}\), locally outside \(z=0\),
\[
{f_U^{1/2}(z)}=z^{n/2}(b_0+b_1 z+b_2 z^2+\cdots).
\]
If $n$ is a zero of even order, putting $n=2s$ we get
\[
{f_U^{1/2}(z)}=z^{s}(b_0+b_1 z+b_2 z^2+\ldots).
\]
In both cases, integrating term by term leads to
\[
F_{0,f_U}(z)=2\int_{0}^z {f_U^{1/2}(\zeta)}d\zeta
=z^{(n+2)/2}(c_0+c_1 z+c_2 z^2+\ldots),
\]
for suitable coefficients $c_i$ with $c_0\neq 0$.
We use Proposition \ref{p:from U to U} to infer that
\begin{equation}\label{Uang}
U(z)=|\Re F_{0,f_U}(z)|
=|c_0|\,|z|^{(n+2)/2}|\cos(\arg z^{ \frac{n+2}{2}})|+o(|z|^{(n+2)/2}).
\end{equation}
The nodal set of \(U\) around \(0\) consists of \(n+2\) analytic 
curves having a common endpoint in the origin with the angle between 
two adjacent rays is equal to \(2\pi/(n+2)\). Hence we infer that the 
origin is a point of multiplicity \(n+2\) for \(U\) and order \(n\) 
for \(f_U\):
\[
m_U(0) = n+2 \quad \textup{and}\quad {\rm ord}(f_U;0)=n.
\]
In particular, we get all the conclusions of the proposition. 
Indeed, i. is a direct consequence of \eqref{formula-x}, because 
in any zero of $f_U$ the order is greater or equal $1$ and, 
therefore, the multiplicity is bigger or equal $3$. The structure 
of $\cN_U$ in ii. and v. are a consequence of \eqref{Uang}. 
Finally, the properties iii. and iv. follows from the fact that 
$U=|\Re F_{z_0, f_U}|$ is locally the modulus of a harmonic 
function away from the critical set, so that changing sign 
across the regular part of $\cN_U$ gives back a harmonic function for which ii. and iii. trivially hold.
\end{proof}

\subsection{Index formula}
We are now ready to state and prove the main result of 
the present section.

We consider functions $U\in \mathcal{U}$ such that $U$ extends 
continuously on $\partial \D$ (with an abuse of notation denoted by $U$ itself) and the number of connected components of
$\{z\in \partial \D : U(z)>0\}$ is an integer $M\in \enne$:
\begin{align}\label{e.sol reg}
M:= \#  \big\{\textup{connected components of } \{z\in \partial \D : U(z)>0\}\big\}.
\end{align}
Set, moreover,
\begin{equation}\label{e.T}
T:=\#  \big\{\textup{connected components of } \overline{\cN_U}\big\},
\end{equation}
with $\overline{\cN_U}$ denoting the closure of the nodal set:
\[\overline{\cN_U}=\{z\in \overline\D : U(z)=0\}.\]

Recalling Remark \ref{r.indeterminato},
from now on we assume that any function $U$ satisfying \eqref{e.sol reg}
is written as $U=(u_1,\ldots,u_N)\in \mathcal{U}_N$ with the condition
\begin{equation}\label{e.sol reg2}
\omega_j =\{u_j>0\} \quad \textup{is connected for }
j\in \{1, \ldots, N\}.
\end{equation}
Set, moreover, for all $\ell, j\in\{1, \ldots, N\}$ with $\ell\neq j$,
\[
\Gamma_{\ell j}:=\pa \om_\ell\cap \pa \om_j \cap \big\{z\in \D: m_U(z)=2\big\},
\]
and let $\{\gamma_j\}_{j=1, \ldots, M}$ be the connected components of $\{U>0\}\cap \partial\D$.


There is a simple connection between the integers $N, M$ and $T$.

\begin{lemma}
Let \(N\geq 2\) and $U\in \cU_N$ satisfying 
\eqref{e.sol reg}-\eqref{e.sol reg2}.  Then,
\begin{equation}\label{e.MNT}
M= N+T-1.
\end{equation}
\end{lemma}

\begin{proof}
The proof is easily done by induction. The case $\overline{\cN_U}$ connected,
i.e. $T=1$, follows from the observation that $\overline{\cN_U}$ is connected
if and only if $\overline \omega_j \cap \partial \D$ is connected, 
i.e. $M=N$.

If the formula is proven for any  $T'< T$, then it follows for $T$. 
Indeed, if $\cN_U$ is not connected, there exists $j_0\in\{1,\ldots,N\}$ 
such that $\overline \omega_{j_0} \cap \partial \D$ is not connected. 
One can then consider a curve $\gamma$ joining two of the connected
components of $\overline \omega_{j_0} \cap \partial \D$ (see Figure \ref{f_index}) and create
two domains $\Omega_1$ and $\Omega_2$ homeomorphic to $\D$ with
\[
\# \{\textup{connected components } \overline{\cN_U}\cap \Omega_\ell\}
=:T_\ell< T \qquad \ell=1, 2.
\]
We can than use the inductive hypothesis and deduce that
\[
M_\ell= N_\ell + T_\ell - 1 \quad \ell=1, 2,
\]
with $M_\ell$ and $N_\ell$ the number of connected components of 
$\{U>0\}\cap \partial \Omega_\ell$ and the number of species in 
$\Omega_\ell$, respectively.
We have that $M=M_1+M_2$, while $N=N_1+N_2-1$, because the set 
$\omega_{j_0}$ intersect both $\Omega_1$ and $\Omega_2$. Summing the two equations
\[
M = M_1+M_2= N_1+N_2+T_1+T_2-2 = N+1 + T- 2 = N+T-1.
\]
\end{proof}
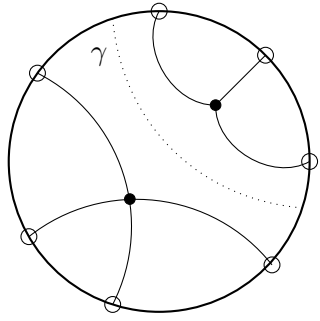
\begin{figure}[h!]
\begin{center}
\begin{tikzpicture}
\draw[thick] (0,0) circle (2);
\coordinate (A) at (0,2);
\coordinate (H) at (1.414,1.414);
\coordinate (E) at (2,0);
\coordinate (B) at (-1.618, 1.17557);
\coordinate (C) at (-1.732, -1);
\coordinate (D) at (-.618, -1.9021);
\coordinate (F) at (1.5,-1.37557);
\coordinate (G) at (.75,.75);
\draw (B) to [bend left=45] (D);
\draw (C) to [bend left=45] (F);
\draw (A) to [bend right=60] (G);
\draw (E) to [bend left=60] (G);
\draw (H) to (G);
\draw[dotted] (-.618,1.902) to [bend right=40] (1.9021,-.618);
\draw (A) circle (3pt);
\draw (B) circle (3pt);
\draw (C) circle (3pt);
\draw (D) circle (3pt);
\draw (E) circle (3pt);
\draw (F) circle (3pt);
\draw (H) circle (3pt);
\filldraw[black] (-.39,-.5) circle (2pt);
\filldraw[black] (G) circle (2pt);
\node at (-.8,1.4) {$\gamma$};
\end{tikzpicture}
\caption{\(M=7, N=6, T=2\); \(M_1=4, N_1=4, T_1=1\), 
\(M_2=3, N_2=3, T_2=1\) .}
\label{f_index}
\end{center}
\end{figure}

We consider the planar graph $(\mathcal V,\mathcal E)$ where the set 
of vertices $\mathcal V$ are
the points in $\cC_U$ (the critical points belonging to the nodal set) 
and the zeros of $U$ on $\partial \D$, the edges $\mathcal E$ are the arcs
$\overline{\Gamma}_{ij}$ and $\overline{\gamma}_j$ with 
$i, j\in \{1, \ldots, N\}$.
For every $z\in \mathcal V$ we define the index
\[
i(z):=m_U(z)-2.
\]
Note that for every $z\in\cC_U$ we have that $i(z)=\textup{ord}(f_U;z)$.
The following result is a generalization of the index formula 
proved in \cite{lm2}.

\begin{proposition}\label{pr:formula}
Let \(N\geq 2\) and $U\in \cU_N$ satisfying \eqref{e.sol reg} and 
\eqref{e.sol reg2}. Then,  the following index formula holds
\begin{equation}\label{formula}
\sum_{z\in {\mathcal V}} i(z) = N-T-1,
\end{equation}
where $T$ is the number of connected components of $\overline\cN_U$.
\end{proposition}

\begin{proof}
%
%
For any \(z\in\cC_U\) the number \(m_U(a)=i(a)+2\) corresponds
to the number of the arcs \(\Gamma_{ij}\) such that 
\(p\in\overline{\Gamma}_{ij}\), whereas for the others vertices 
\(z\in\pa\mathbb{D}\) it holds that \(m_U(z)+1\) is the number of
the arcs \(\Gamma_{ij}\) such that \(z\in\overline{\Gamma}_{ij}\). Then,
\begin{align*}
 \sum_{z\in \mathcal V} i(z)
& = \sum_{z\in\mathcal{C}_U} \left[m_U(z)-2\right]
  +\sum_{ z\in \mathcal V\cap \pa\mathbb{D}}\left[m_U(z)-2\right] \\
& = \sum_{z\in\mathcal{C}_U}\left[\#\{e\in \mathcal E: z\in {e}\}-2\right]
  +\sum_{z\in \mathcal V\cap \pa\mathbb{D}}\left[\#\{e\in \mathcal E: z\in e\}-3\right] \\
& = \sum_{p\in \mathcal V}
  \left[\#\{e\in \mathcal E: z\in e\}-2\right] -M\\
&=2\# \mathcal E-2\# \mathcal V-M=2\# \mathcal F-4-M\\
&=2(N+1)-4 - M\stackrel{\eqref{e.MNT}}{=}2N-2 -N-T+1 = N-T-1,
\end{align*}
where we used the Euler formula $\# \mathcal E-\# \mathcal V=\# \mathcal F-2$, with $\mathcal F$ the number of faces which equals $N+1$ (the nodal regions plus the complement of $\D$), and we used that each edge is counted twice in the last sum (once for each endpoint).
\end{proof}

\begin{remark}
A consequence of Proposition \ref{pr:formula} is that for every $U\in \cU_N$ satisfying \eqref{e.sol reg} and \eqref{e.sol reg2} one has \(m_U(z)\leq N\) for all \(z\in \mathcal{C}_U\).
\end{remark}

%

\section{Untie singular points}\label{s:desingularization}
In this section we provide the main step of the proof of genericity in $\cU$ of functions with only
triple junctions.
Here we consider functions $U\in \cU$ satisfying the  following conditions:
\begin{itemize}
\item[(H1)] The associate holomorphic function $f_U:=I(U)$ extends 
to a holomorphic function in a neighborhood of the unit disk $\D$ 
and therefore its zero set is finite in $\D$: 
\[
\cC_U =\cZ_{f_U}\cap \mathcal{N}_U
=\big\{z_0, z_1, \ldots, z_M\big\}, \qquad M\in \enne.
\]
\item[(H2)] The order of $z_0$ is bigger than $1$:
\[
\textup{ord}(f_U; z_0) = m_0+1\qquad m_0>0,
\]
\item[(H3)] The zeros $z_j$ are in \textit{general position} with 
respect to $z_0$, meaning that\\
i. the distances $|z_j-z_0|$ are all different,\\
ii. the half lines $\Gamma_j=\{z_j+t(z_j-z_0):t\geq 0\}$ does
not intersect $\cZ_{f_U}$.
\end{itemize}

Under assumptions (H1)-(H3) we show how to perturbe $f_U$ in 
such a way to find a new segregated state $V$ close to $U$ and a point $\omega_0$ close to $z_0$ satisfying
\[
\cC_V= \{\omega_0, z_0, \ldots, z_M\},
\]
and
\begin{gather*}
\textup{ord}(f_V;\om_0)=1, \quad \textup{ord}(f_V;z_0)=m_0,\\
\textup{ord}(f_V;z_j)= \textup{ord}(f_U;z_j) \quad \forall\; j=1,\ldots, M.
\end{gather*}

Before entering the details of the proof of the main result of the
present section, we give, in subsection \ref{ex.no U}, an example 
which shows the kind
of ``desingularization" one can observe in Figure \ref{figtype2}, 
although possible, can only happen by choosing the directions
connecting the newly generated critical points in a very rigid way.

\subsection{An example of rigidity}\label{ex.no U}
The set of holomorphic functions with
simple zeros is trivially generic. However, the
subset of holomorphic functions belonging to $I(\cU)$ is highly non-generic.
This fact can be appriciated by a simple example. Consider the harmonic 
function \(U\) and its holomorphic Hopf differential
\[
U(z) = \left|\Re\left(\tfrac{2}{5} z^{5/2}\right)\right|,
\qquad
I(U) =U_z^2 = f = \tfrac{1}{4} z^3,
\]
In order to untangle the critical point \(z=0\), we have to consider the
holomorphic function $f_w$, depending on a complex parameter written in polar for as $w=\eps\ee^{i\phi}\in \ci$,
\[
f_w(z)=\tfrac{1}{4}z (z-w)^2.
\]
Clearly, \(\cZ_{f_w}=\{0,w\}\). Consider the square root of \(f_w\)
defined in the simply connected open set 
\(\Omega^-:=\{z=\rho e^{i\theta}\in \ci : 0<\rho<1,  \theta\in (-\pi,\pi)\}\)
\[
f_w^{1/2}(z)=\tfrac{1}{2} z^{1/2} (z-w),
\]
where \(z^{1/2} = \sqrt{\rho} e^{i\frac{\theta}{2}}\) for every 
\(z=\rho e^{i\theta}\) with \(\theta\in (-\pi,\pi)\).
The primitive of \(f_w^{1/2}\) which vanishes in \(0\) is given by
\begin{align*}
 F_w(z)
&= 2\int_0^z {\zeta}^{1/2} (\zeta-w)d\zeta
   =\tfrac{2}{5} z^{5/2}-\tfrac23 z^{3/2}w \\
&=\tfrac{2}{5} \rho^{5/2} e^{i 5\theta/2}-\tfrac{2}{3}\eps \rho^{3/2}
   e^{i\left(3\theta/2+\phi\right)}.
\end{align*}
The trace of \(\Re F_w\) on the boundary of \(\D\) is
\[
G_\phi(\theta)=\tfrac{2}{5}\cos\left(\tfrac{5}{2}\theta\right)
  -\tfrac{2}{3} \eps\cos\left(\tfrac{3}{2}\theta+\phi\right).
\]
It is easy to verify that, for $\eps>0$ sufficiently small, \(G_\phi\) vanishes
exactly in five points on \(\partial\D\). For \(\eps=0\),
\(G_0(\theta) = 0\) if and only if \(\theta=\pi/5 +2k\pi/5\) for \(k=0, \ldots, 4\).

Set now \(U_w(z)=|\Re F_w(z)|\). If \(w=0\), then \(U_0\in\cU_5\) and it is the solution with boundary value \(G_0\) (see Figure \ref{figuraex}, left).
However, for $\eps>0$ this is almost never the case.
Notice that
\[
U_w(0)=0 \quad\textup{and}\quad U_w(w)
=\tfrac{4}{15} |w|^{5/2}\left|\cos\big(\tfrac52 \phi\big)\right|.
\]
Hence, \(U_w(w)=0\) if and only if \(\phi=\frac{\pi}{5} + \frac{2k\pi}{5}\)
for \(k=0, \ldots, 4\). In this five instances, Proposition \ref{p:from f to U} 
implies that \(U_w\in \cU_5\) and \(f_{U_w} = f_w\) (see Figure \ref{figuraex}, right, to help intuition). In all the other cases,
i.e. \(\phi\neq \frac{\pi}{5}+\frac{2k\pi}{5}\),  we have that
\(w\in\mathcal{Z}_{f_{U_w}}\) but \(w\not\in\cN_{U_w}\), and 
hence \(\cC_{U_w}=\mathcal{Z}_{f_{U_w}}\cap\cN_{U_w}=\{0\}\). 
From the index formula \eqref{formula} we infer that, if \(U_w\in\cU_5\), 
then the origin should be a point with multiplicity \(5\). This is however 
impossible because 
\[
U_w(z)=|\Re F_w(z)|=\left|\tfrac25 \rho^{5/2} \cos (\tfrac52 \theta)
-\tfrac23 \rho^{3/2} \epsilon \cos(\tfrac32 \theta +\phi_w)\right|
= O(\rho^{3/2})
\]
contradicts \eqref{local} in Proposition \ref{p.reg2}.
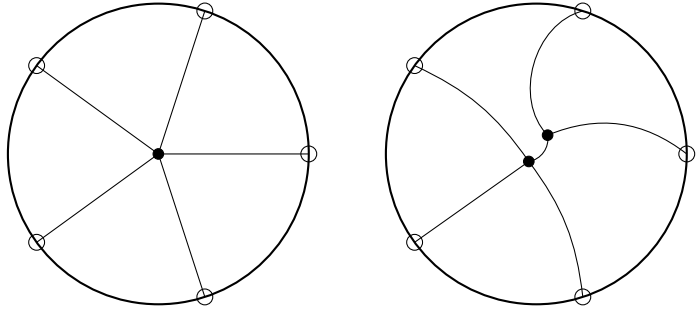
\begin{figure}[h!]
\begin{center}
\begin{tikzpicture}
\draw[thick] (0,0) circle (2cm);
  \coordinate (A) at (0.618, 1.90211);
    \coordinate (B) at (-1.61803, 1.17557);
     \coordinate (C) at (-1.61803, -1.17557);
    \coordinate (D) at (0.618, -1.90211);
     \coordinate (E) at (2,0);
\draw (0,0) -- (A);
\draw (0,0) -- (B);
\draw (0,0) -- (C);
\draw (0,0) -- (D);
\draw (0,0) -- (E);
\draw (A) circle (3pt);
\draw (B) circle (3pt);
\draw (C) circle (3pt);
\draw (D) circle (3pt);
\draw (E) circle (3pt);
\filldraw[black] (0,0) circle (2pt);
\end{tikzpicture}
\qquad
\begin{tikzpicture}
\draw[thick] (0,0) circle (2cm);
\coordinate (A) at (0.618, 1.9021);
\coordinate (B) at (-1.618, 1.17557);
\coordinate (C) at (-1.618, -1.17557);
\coordinate (D) at (0.618, -1.9021);
\coordinate (E) at (2,0);
\coordinate (P1) at (0.15,0.25);
\coordinate (P2) at (-.1,-.1);
\draw (P1) to [bend left=60] (A);
\draw (P1) to [bend left=30] (E);
\draw (P1) to [bend left=45] (P2);
\draw (P2) -- (C);
\draw (P2) to [bend right=15] (B);
\draw (P2) to [bend left=15] (D);
\draw (A) circle (3pt);
\draw (B) circle (3pt);
\draw (C) circle (3pt);
\draw (D) circle (3pt);
\draw (E) circle (3pt);
\filldraw[black] (P1) circle (2pt);
\filldraw[black] (P2) circle (2pt);
\end{tikzpicture}
\end{center}
\caption{Five species: configuration with one 5-point (on the left); 
configuration with one 3-point and one 4-point (on the right).}
\label{figuraex}
\end{figure}

As we will see in the next result, the desingularization
is a global process which involves simultaneously all critical
points.

\subsection{The main lemma}
The Main Theorem on the multiplicity of points $z\in \Omega^-$ for 
the generic solutions $U$ to \eqref{e.dirichlet energy} stated in the 
Introduction requires the following lemma, which proves, roughly speaking, that there exists a suitable small 
perturbation of the function \(U\), which decrease the order of one zero of $f_U$,  leaving all the other unchanged.

For the readers' convenience, we illustrate the main steps of the proof:
\begin{itemize}
\item Step 1: given $U\in \cU$, we consider the corresponding Hopf differential $f_U$ and propose a modification of it which consists in multiplying by a polynomial of several complex variables $w_\ell$ with coefficients holomorphic functions $h_\ell(z)$. 

\noindent The main idea of Step 1 is the introduction of new complex parameters $w_\ell$ which will be chosen in such a way to satisfy the necessary algebraic conditions  \eqref{e.parti reali} in order to produce a new Hopf differential.

\item Step 2: we study the linear system in the new variables  introduced in Step 1 corresponding to the solutions of \eqref{e.parti reali}. A convenient choice for the holomorphic coefficients $h_\ell(z)$ is a sequence of carefully chosen powers $z^{k_\ell}$ with $k_\ell := \ell R$ for a large constant $R$.

\item Step 3 shows that the choices in the previous step allow to prove the invertibility of the linear system guaranteeing \eqref{e.parti reali} in all points except one (corresponding to the new zero of the holomorphic perturbation).
The solution of Step 3 depends on the single parameter corresponding to the position of the new zero of the holomorphic perturbation.

\item Step 4 contains the implicit function argument which ensures that the algebraic equation \eqref{e.parti reali} can be satisfied also in the new zero, while preserving all the other solutions. This step takes care of the rigidity presented in \S \ref{ex.no U}, showing that only finitely many directions are admissible for this final perturbation.

\item The final Step 5 shows that the perturbation we make is arbitrarily small in $H^1(\D)$ and $L^\infty(\D)$. 
\end{itemize}

\begin{lemma}\label{l.perturbazione}
Let $U\in \cU$ with $f_U$ satisfying (H1)-(H3). Then, for every 
$\eps>0$ there exists $U_\eps\in \cU$ such that
\[
\|U-U_\eps\|_{L^\infty(\D)} + \|U-U_\eps\|_{W^{1,2}(\D)}\leq \eps,
\]
and
\begin{gather*}
\cC_{U_\eps}= \{\omega_0, z_0, \ldots, z_M\},\\
\textup{ord}(f_{U_\eps};\om_0)=1, \qquad
  \textup{ord}(f_{U_\eps};z_0)=m_0,\\
\textup{ord}(f_{U_\eps};z_j)= \textup{ord}(f_U;z_j)
  \qquad \forall j=1,\ldots, M.
\end{gather*}
\end{lemma}

\begin{proof}
We consider the simply connected domain
\[
\Omega^-:= \D \setminus \bigcup_{j=0}^M\Gamma_j,
\]
where $\Gamma_j=\{z_j+t(z_j-z_0):t\geq 0\}$ for $j=1, \ldots, M$ 
and
\[
\Gamma_0:=\{z_0+t\gamma_0: t\geq 0\},\quad \gamma_0\in\ci, \; |\gamma_0|=1
\]
is any half-line not
intersecting $\cup_{j=1}^M\Gamma_j$.
We translate our domain by $-z_0$:
\[
\Omega_0:= \big\{z\in \ci: z+z_0 \in \Omega^-\big\}.
\]
The functions
\[
U_0(z):= U(z-z_0), \quad f_0(z):= f_U(z-z_0)
\]
are defined in $\D_0:=\big\{z+z_0 : |z|<1\big\}$ and satisfy
\[
\cC_{U_0} = \big\{0, \om_1, \ldots, \om_M\big\} 
\quad\textup{with}\quad \om_j:=z_j-z_0 \quad \forall j=1,\ldots, M.
\]
The set $\Omega_0$ is star-shaped and we have that
\[
\Omega_0 = \D_0 \setminus \left(\bigcup_{j=1}^M 
\{t\, \om_j: t\geq 1\}\cup \{t\gamma_0: t\geq 0\}\right).
\]
Note that by assumption (H3) we have that
\begin{equation}\label{e.general}
0<|\om_1|<\cdots <|\om_M|,
\end{equation}
and
\[
\frac{\om_j}{|\om_j|} \neq \frac{\om_\ell}{|\om_\ell|} 
\quad \textup{and} \quad \frac{\om_j}{|\om_j|} \neq \gamma_0 
\qquad \forall j\neq \ell \in \{1, \ldots, M\}.
\]
For the sake of readability we split the proof into different steps.
\medskip

\noindent\textbf{Step 1: modified holomorphic function.}
By (H1) - (H2) the function \(f_0\) can be written in the form
\begin{equation}\label{def.f2}
f_0(z)=z^{m_0+1}\, h^2(z) \, \prod_{j=1}^{M} (z-\om_j)^{q_j},
\qquad q_j\in\enne, \;j=1, \ldots, M,
\end{equation}
with $h:\D_0\to \ci$ a holomorphic function with $h(\omega_j)\neq 0$ for all $j$.
We fix $M$ holomorphic functions \(h_1,...,h_M\) in $\D_0$ to be 
specified later and for every vector \(W=(w_1,...,w_M)\in\mathbb{C}^M\) 
we set
\begin{equation}\label{yyy}
q(z,W):=1+\sum_{\ell=1}^M w_\ell h_\ell(z)\qquad z\in \D_0.
\end{equation}
Let \(\om_0\in \C\) with \(0<|\om_0|<|\om_1|\) and consider the function
\begin{equation}\label{xxx}
f_{\omega_0,W}(z):= z^{m_0} (z-\om_0)\,h^2(z)\,q^2(z,W)\, 
\prod_{j=1}^{M} (z-\om_j)^{q_j} .
\end{equation}
We consider a simply connected domain 
$\Omega_{\omega_0}:= \Omega_0\setminus L_{\omega_0}$, with 
$L_{\omega_0}:=\{\omega_0+t\gamma_0 : t\geq 0\}$. On the simply connected 
domain $\Omega_{\omega_0}$ we consider a determination of the square root
\(f_{\om_0,W}^{1/2}: \Omega_{\om_0}\to \mathbb{C}\) written in the 
following form:
\begin{align*}
f_{\om_0,W}^{1/2}(z)
&:=z^{m_0/2}\,(z-\om_0 )^{1/2}\, H(z)\,q(z,W) \\
&=z^{m_0/2}\,(z-\om_0 )^{1/2}\, H(z)+z^{m_0/2} \,(z-\om_0 )^{1/2}\, 
  H(z)\sum_{\ell=1}^M w_\ell\, h_\ell(z) ,
\end{align*}
with
\[
H(z):=\prod_{j=1}^M (z-\om_j)^{{q_j}/{2}} h(z),
\]
where we fixed suitable determinations of the square roots $z^{m_0/2}$, 
$(z-\omega_0)^{1/2}$ and $(z-\omega_j)^{q_j/2}$ in $\Omega_{\om_0}$.
We set
\begin{align*}
{F}_{\om_0, W}(z)&:=2\int_0^z \z^{m_0/2} (\z-\om_0)^{1/2}
  H(\z) d\z \\
&\qquad +2\sum_{\ell=1}^M w_\ell \int_0^zh_\ell(\z)
  \z^{m_0/2} (\z-\om_0)^{1/2} \,H(\z)\, d\z,
\end{align*}
where the integration is taken along curves in $\Omega_{\om_0}$.
By construction, we have that \({F}_{\om_0, W}(0)=0\). Define the matrix 
\(A=A(\om_0 )=\{A_{j,\ell}(\om_0)\}\in \ci^{M\times M}\) as follows
\[
A_{j,\ell}=A_{j,\ell}(\om_0 )
:=2\int_0^{\omega_j}h_\ell(\z) \z^{m_0/2} (\z-\om_0)^{1/2} \,H(\z)\, d\z
 \qquad j,\ell=1,...,M,
\]
and the vector \(B=B(\om_0 )\in \ci^M\) with components
\[
b_j=b_j(\om_0 ):=2\int_0^{\omega_j} \zeta^{m_0/2} 
(\zeta-\om_0 )^{1/2}\, H(\zeta)\, d\zeta
\qquad j=1,...,M.
\]
With these notations
\begin{equation}\label{14}
{F}_{\om_0, W}(\om_j)=b_j+\sum_{\ell=1}^M A_{j,\ell} w_\ell
=(A W +B)_j \qquad j=1,...,M.
\end{equation}

\noindent\textbf{Step 2: invertibility of $A$.}
We show that one can choose \(h_1,...,h_\ell\in \mathcal{H}(\ci)\) 
such that the matrix \(A(0)\) is non-singular and, hence, by continuity 
the same holds for \(A(\om_0)\) if $|\om_0|$ is sufficiently small.

Consider the case
\[
h_\ell(z)=z^{k_\ell},\quad\textup{with}\quad k_1<k_2<...<k_M,
\quad  k_\ell\in\enne.
\]
We integrate along the path \(\zeta=t \om_j\) with \(t\in [0,1]\),
\begin{align*}
  A_{j,\ell}(0)
& =2\int_0^{\omega_j} \z^{\frac{m_0}{2}+\frac12+k_\ell} \prod_{s=1}^M 
  (\z-\om_s)^{q_s/2} h(\z)d\z\\
& =2(\om_j)^{\frac{m_0}{2}+\frac32+k_\ell} \int_0^1 t^{\frac{m_0}{2}
  +\frac{1}{2}+k_\ell}\prod_{s=1}^M (t \om_j-\om_s)^{q_s/2}
  h(t\om_j)dt\\
& =2(\om_j)^{\frac{m_0}{2}+\frac32+Q+k_\ell}  \int_0^1 
  t^{\frac{m_0}{2}+\frac12+k_\ell} \prod_{s=1}^M 
  \left(t -\frac{\om_s}{\om_j}\right)^{q_s/2} h(t\om_j)dt,
\end{align*}
with \(Q=(q_1+\ldots+q_M)/2\). Set
\begin{align*}
  a_{j,\ell}
& =\int_0^1 t^{\frac{m_0}{2}+\frac12+k_\ell}\prod_{s=1}^M 
  \left(t -\frac{\om_s}{\om_j}\right)^{q_s/2} h(t\om_j)dt\\
& =i^{q_j}\int_0^1 t^{k_\ell}(1-t)^{q_j/2} \, t^{\frac{m_0}{2}+\frac12}
  \prod_{s\neq j} \left(t -\frac{\om_s}{\om_j}
  \right)^{q_s/2} h(t\om_j)dt\\
& =\int_0^1 t^{k_\ell}(1-t)^{q_j/2} g_{j}(t)\, dt,
\end{align*}
with
\[
g_{j}(t) := i^{q_j} t^{\frac{m_0}{2}+\frac12}\prod_{s\neq j}
  \left(t -\frac{\om_s}{\om_j}\right)^{q_s/2} h(t\om_j).
\]
Then,
\[
A_{j,\ell}(0)=2(\om_j)^{\frac{m_0}{2}+\frac32+Q+k_\ell} \, a_{j,\ell},
\qquad j,\ell=1,...,M.
\]
The matrix \(A\) is non singular if and only if the following matrix 
is non-singular:
\[
{\mathcal A}=\begin{array}({cccc})
\om_1^{k_1} a_{11} & \om_1^{k_2} a_{12} & ... & \om_1^{k_M}a_{1M} \\ \\
\om_2^{k_1} a_{21} & \om_2^{k_2} a_{22} & ... & \om_2^{k_M}a_{2M} \\
\vdots & \vdots & & \vdots \\
\om_M^{k_1} a_{M1} & \om_M^{k_2} a_{M2} & ... & \om_M^{k_M}a_{MM} \\
\end{array}.
\]
We assume \(k_\ell=\ell\,R\) with \(R>0\) and claim that:
\[
\det \mathcal{A} = \begin{array}|{cccc}|
\om_1^{R} a_{11} & \om_1^{2R} a_{12} & \ldots & \om_1^{M R}a_{1M} \\ \\
\om_2^{R} a_{21} & \om_2^{2R} a_{22} & \ldots & \om_2^{M R}a_{2M} \\
\vdots & \vdots & & \vdots \\
\om_M^{R} a_{M1} & \om_M^{2R} a_{M2} & \ldots & \om_M^{M R}a_{MM} \\
\end{array}\neq 0
\]
for $R$ sufficiently large.

To this aim, we start looking at the functions
\[
m_{j\ell}(t):= t^{lR}\, (1-t)^{q_j/2}.
\]
By an explicit computation we have that
\begin{align*}
M_{j\ell}&:=\int_0^1 m_{j\ell}(t)\, dt 
  = \dfrac{\Gamma(1+\ell R)\,\Gamma(1+q_j/2)}{\Gamma(2
    +\ell R +q_j/2)},
\end{align*}
where $\Gamma$ is Euler's Gamma function. In particular, by the 
well-known relations $\Gamma(1+x)=x\Gamma(x)$ and 
$\Gamma(x)\Gamma(x+\frac12)=2^{1-2x}\sqrt{\pi} \,\Gamma(2x)$,
we infer that
\[
M_{j\ell} = (\ell R)^{-1-q_j/2}\big(1+o(1)\big)
\quad \hbox{as }R\to +\infty.
\]
Moreover, we have that
\begin{equation}\label{e.delta}
\mu_{j\ell}:=\frac{m_{j\ell}(t)}{M_{j\ell}} dt  
\stackrel{\star}{\rightharpoonup} \delta_1.
\end{equation}
Indeed, $\mu_{j\ell}$ are probability measures such that for every 
$\tau<1$ we have
\begin{align*}
  \mu_{j\ell}([0, \tau])
& =M_{j\ell}^{-1}\int_0^\tau t^{\ell R}(1-t)^{q_j/2} dt
  \leq M_{j\ell}^{-1}\int_0^\tau t^{\ell R}dt\\
&\leq \frac{\tau^{\ell R}}{(\ell R +1)\, (\ell R)^{1+q_j/2}
  (1+o(1))}\to 0 \qquad\hbox{as }R\to +\infty
\end{align*}
This implies that any weak$^\star$  limit $\mu$ of $\mu_{j\ell}$ 
is zero in $[0,\tau]$ for every $\tau<1$. Since the space of 
probability measures on a compact set is weak$^\star$ closed, 
we infer \eqref{e.delta}. In particular, we have that
\begin{align}\label{e.asintotica a}
  \lim_{R\to +\infty} (\ell R)^{1+q_j/2} a_{j\ell}
& = \lim_{R\to +\infty}M_{j\ell}^{-1}\,a_{j\ell}\notag\\
& =\lim_{R\to +\infty}M_{j\ell}^{-1}\int_0^1 t^{k_\ell}
  (1-t)^{q_j/2} g_{j}(t)\, dt= g_j(1).
\end{align}
We can now conclude the proof of the invertibility of $\mathcal{A}(0)$.
We use the Leibnitz formula for the determinants
\begin{align*}
  \det {\mathcal A}
& = \sum_{\sigma}\hbox{sgn}(\sigma) \prod_{j=1}^M  
  \om_j^{\sigma(j) R}a_{j \sigma(j)}\\
& = \om_1^R \om_2^{2R}
  \cdot ... \cdot\om_M^{M R} \Bigg[a_{11} ... a_{MM} \\
&\qquad +\sum_{\sigma\neq (1,2, ... ,M)} \hbox{sgn}(\sigma)
  \prod_{j=1}^M  a_{j \sigma(j)} \dfrac{\om_1^{\sigma(1) R}
  \om_2^{\sigma(2) R}\cdots\om_M^{\sigma(M) R}}
  {\om_1^R \om_2^{2R}\cdots\om_M^{M R}} \Bigg] ,
\end{align*}
where as usual \(\sigma=(\sigma(1),...,\sigma(M))\) is a permutation 
of \((1,...,M)\) and \({\rm sgn}(\sigma)\) denotes its sign.
Using
\eqref{e.asintotica a} in the form
\[
a_{j\ell} = (\ell R)^{-1-q_j/2}g_j(1) (1+o(1)),
\]
we get (recall that $Q=q_1/2+\ldots +q_M/2$)
\begin{align*}
  \det {\mathcal A}
& =\om_1^R \om_2^{2R}  \cdots \om_M^{M R} R^{-M-Q}
  \prod_{j=1}^M g_j(1) \Bigg[\prod_{j=1}^M j^{-1 -q_j/2} \\
& +\sum_{\sigma\neq (1,2,\ldots,M)}
  \hbox{sgn}(\sigma) \prod_{j=1}^M \sigma(j)^{-1-q_j/2}    
  \dfrac{\om_1^{\sigma(1) R}\om_2^{\sigma(2) R} \cdots
  \om_M^{\sigma(M) R}}{\om_1^R \om_2^{2R} \cdots \om_M^{M R}}
  +o(1) \Bigg].
\end{align*}
By the assumption (H3) on the zeros of $f_U$ in general position (i.e. \eqref{e.general}), we deduce that
\[
\lim_{R\to \infty} \frac{\om_1^{\sigma(1) R}\om_2^{\sigma(2) R}\cdots \om_M^{\sigma(M) R}}{\om_1^R \om_2^{2R}  \cdots\om_M^{M R}}=0,\qquad \forall \sigma\neq (1,2,...,M).
\]
Then, taking into account that $g_j(1)\neq 0$, it follows 
that  $\det\mathcal{A}\neq0$ for sufficiently large $R$ 
(depending only on the exponents $q_j$'s), thus concluding 
the proof of the claim.

\medskip

\noindent \textbf{Step 3: $\Re F_{\om_0,W}(\om_j)=0$ for $j=1, \ldots, M$.}
We prove that it is possible to find a column vector \[W=W(\om_0 )=\big(w_1(\om_0 ),...,w_M(\om_0 )\big),\]
such that the conditions \eqref{e.parti reali} are satisfied.
Keeping in mind \eqref{14}, conditions \eqref{e.parti reali} can be written
in the single equation
\begin{equation}\label{16}
A(\om_0) W(\om_0 ) + B(\om_0 )=i \Lambda(\om_0 ),
\end{equation}
with an arbitrary \(\Lambda(\om_0)\in \erre^M\).
Note that, since for $\om_0=0$ we have that $f_{0,0}=f_0$,
then $F_{0,0}$ satisfies \eqref{e.parti reali}, i.e.,
\begin{align*}
\big(F_{0,0}(\om_j)\big)_{j=1, \ldots, m} =: B(0) \in i\, \erre.
\end{align*}
Since the matrix \(A\) is non-singular, the system \eqref{16} is then solved by
\begin{equation}\label{W}
W(\om_0 )=(A(\om_0 ))^{-1}(B(0)-B(\om_0 )).
\end{equation}
In particular, the function $\om_0\mapsto W(\om_0)$ is continuous and
\[
\lim_{\om_0\to 0} W(\om_0)=W(0)=0.
\]

\medskip

\textbf{Step 4: $\Re {F}_{\om_0, W(\om_0)}(\om_0 )=0$.}
We show that it is possible
to choose $\om_0$
such that also the real part of ${F}_{\om_0, W(\om_0)}(\om_0)$ 
is zero.

In order to compute $F_{\om_0, W(\om_0)}$ we integrate along 
the path \(\z=t \om_0\), \(0\leq t \leq 1\) and we get
\[
F_{\om_0, W(\om_0)}(\om_0)=\int_0^1 (t\om_0)^{m_0/2} 
((t-1)\om_0)^{1/2} H(t\om_0) q(t\om_0, W(\om_0))\om_0 dt.
\]
We consider \(\om_0=\eps \ee^{i \vth}\) and
\begin{align*}
G(\eps,\vth)
&:=\Re F_{\eps \ee^{i \vth}, W(\eps \ee^{i \vth})}
(\eps \ee^{i \vth}) \\
&=\eps^{\frac{m_0+3}{2}}\Re \left(i\ee^{\frac{i (m_0+3) \vth}{2}}
\int_0^1 t^{m_0/2}\sqrt{1-t} Z(t\eps \ee^{i \vth}) dt\right),
\end{align*}
with
\[
Z(t\eps \ee^{i \vth}) 
:= H(t\eps \ee^{i \vth})
q(t\eps \ee^{i \vth},W(t\eps \ee^{i\vth})).
\]
Set
\[
K(\eps,\vth):=\eps^{-\frac{m_0+3}{2}} G(\eps,\vth)
= \Re\left(i\ee^{\frac{i (m_0+3) \vth}{2}} \int_0^1 
t^{m_0/2}\sqrt{1-t} Z(t\eps \ee^{i \vth})\, dt\right).
\]
We show that the equation \(K(0,\vth)=0\) has a finite number of 
roots. Indeed, we write $Z(0)=|Z(0)| \ee^{-i \varphi}$ and
\[
c_{m_0}=\int_0^1 t^{m_0/2} \sqrt{1-t}dt
=\frac{\sqrt{\pi}}{2} 
\dfrac{\Gamma\big(1+\frac{m_0}{2}\big)}{\Gamma\big(\frac{5+m_0}{2}\big)}.
\]
Then
\[
K(0,\vth) = |Z(0)| c_{m_0} 
\Re\left(i e^{i\big( \frac{(3+m_0)\vth}{2}-\varphi \big)}\right)
= -|H(0)| c_{m_0} \sin\left(\frac{3+m_0}{2}\vth-\varphi \right).
\]
Since $H(0)\neq 0$, we infer that
\begin{align*}
K(0,\vth)=0 &\iff \sin \left( \frac{3+m_0}{2}\vth-\varphi \right)=0\\
& \iff  \frac{3+m_0}{2}\vth-\varphi=k\pi,\quad k\in\zzz,
\end{align*}
that is the zeros of $K(0,\vth)$ are
\[
\vth_k=\frac{2\varphi+2k\pi}{3+m_0},\quad k=0,...,2+m_0.
\]
Let us fix an index \(k\in\{0,...,2+m_0\}\).
We have
\begin{align*}
\frac{\partial K(0,\vth)}{\partial \vth}\Big|_{(0,\vth_k)}
=- \frac{3+m_0}{2}c_{m_0}|H(0)| 
\cos\left( \frac{3+m_0}{2}\vth_k-\varphi \right)\neq 0,
\end{align*}
because $H(0)\neq 0$.
Given the regularity of $K(\eps, \vth)$, for the implicit function theorem 
there exists a function \(\vth_k(\eps)\), \(\eps\in(-\eps_0,\eps_0)\) with
sufficiently small \(\eps_0>0\), such that
\[
\vth_k(0)=\vth_k,\qquad K(\eps,\vth_k(\eps))=0.
\]
It follows that there exist a finite number of directions \(\vth_k(\eps)\)
such that, if we set \(\om_0(\eps)=\eps \ee^{i \vth_k(\eps)}\), condition
\[
\Re F_{\om_0,W(\om_0)} (\om_0) = 0
\]
is satisfied.

\medskip

\textbf{Step 5: conclusion of the proof.}
For $\eps>0$ sufficiently small, we consider a solution 
$\om_0(\eps)$. Then, 
$f_{\eps}(z):=f_{\om_0(\eps), W(\eps)}(z+z_0)\in \mathcal{H}$.
By construction we have that 
$\{z_0, z_0+\om_0, z_1,..., z_M\}\subseteq \mathcal{Z}_{f_\eps}$ 
and
\[
\textup{ord}(f_\eps; z_j)=\textup{ord}(f_U; z_j) 
\quad \forall\; j=1,..., M,
\]
and $\textup{ord}(f_\eps; z_0)=m_0$ and 
$\textup{ord}(f_\eps; z_0+\om_0)=1$.
Since $W(\om_0(\eps))\to W(0)=0$, we have also that
\begin{equation}\label{e.vicinanza f}
\lim_{\eps\to 0^+}\|f_\eps - f_{U}\|_{L^\infty(\D)} =0.
\end{equation}
Moreover, from $\cZ_{f_U}^\textup{odd}\subset \cZ_{f_U}\cap \cC_U$, we have
\(
\cZ_{f_\eps}^\textup{odd}\subset\{\omega_0, z_0, \ldots, z_M\}
\).
Set $F_{\eps}:= F_{\om_0(\eps), W(\om_0(\eps))}$. By Steps 3 and 4 $\Re F_{\eps}(z)=0$
for all $z\in \cZ_{f_\eps}^\textup{odd}$. By Proposition \ref{p:from f to U} to infer that there exists a function $U_\eps=|\Re F_{\eps}|\in \mathcal{U}$ such that $f_\eps = I( U_\eps)$.

Notice that, since $U_\eps (z_0)=0$ and \eqref{e.vicinanza f} holds,
we have $U_\eps\to U$ uniformly. Furtheromore, $|\nabla U_\eps|^2 \to |\nabla U|^2$ in $L^1$, which follows from $f_\eps\to f_U$ in $L^1$. In particular,
$\nabla  U_\eps \rightharpoonup \nabla U$ in $L^2$ and the Dirichlet energies converge, thus implying $U_\eps\to U$ in $H^1(\D)$.

%
%
\end{proof}

\section{Generic segregated states}\label{s:generic}

In this section we prove Theorem \ref{t:main} on the genericity of functions $U\in \cU$ with only triple junctions.
By \eqref{formula-x}, we need to prove
that the subset of functions $U\in \mathcal{U}$ such that $I(U)$ 
has only simple zeros is the intersection of open dense subsets.

\subsection{Finite number of critical points}
First of all we show that the functions with only finitely many 
critical points constitute a dense set of $\cU$.

\begin{lemma}\label{l.density}
The set of $U\in \cU$ such that the corresponding  holomorphic 
functions $f_U$ have finitely many zeros is dense in $\cU$.
\end{lemma}

\begin{proof}
Given any function $U\in\cU$, it is enough to consider the functions
\[
U_\eps(z):=U\left(\frac{z}{1+\eps}\right) \qquad z\in \D.
\]
Then, clearly $U_\eps\in \mathcal{U}$ and its squared Wirtinger 
derivative satisfies
\[
f_{U_\eps}(z):=\big(\partial_z U_\eps\big)^2 = (1+\eps)^{-2} f_U
\left(\frac{z}{1+\eps}\right).
\]
Therefore, since the set of its zeros of the holomorphic function 
$f_U$ is locally finite on compact subsets of $\D$, we deduce 
in particular that
\[
\#\cZ_{f_{U_\eps}}
=\#\big(\cZ_{f_{U}}\cap \D_{\frac1{1+\eps}}\big)<+\infty.
\]
Moreover, by the very definition we also have that
\[
\|U_\eps -U\|_{H^1(\D)}\to 0 \qquad \textup{as}\qquad  \eps\to 0.
\]
\end{proof}

\begin{remark}
Considering that $\mathcal{C}_{U_\eps}\subseteq \mathcal{Z}_{f_{U_\eps}}$, we conclude that the set of functions with finitely many critical points is dense in $\cU$.
\end{remark}

\subsection{Points in general position}
In the perspective of proving a density result, we can then assume that the number of critical points of \(U\in \mathcal{U}\) is finite, say
\[
\mathcal{C}_U = \big\{z_0, z_1, \ldots, z_M\big\}, \quad M\in \enne.
\]
We show in the next lemma that there exists a conformal diffeomorphism $\phi$ of the unit disk such that the image of the critical set is in general position with respect to $\phi(z_0)$ as defined in (H3) in Section \ref{s:desingularization}.


We recall that the conformal diffeomorphisms of the unit disk are characterized in terms of two parameters
\[
\phi_{\alpha,\theta}(z)=
e^{i\theta}\dfrac{z+\alpha}{\overline{\alpha}z+1}
\qquad \alpha\in\mathbb{D}, \quad \theta\in\erre.
\]
When $\theta=0$ we write $\phi_{\alpha}$ in place of $\phi_{\alpha, 0}$.

\begin{lemma}\label{l.general position}
Let $\{p_0, \ldots, p_M\}\subset\D$.
Then, there exists $\alpha\in \D$ such that
\[
\big\{\phi_{\alpha}(p_0), \ldots,\phi_{\alpha}(p_M)\big\}
\]
are in general position with respect to $\phi_{\alpha}(p_0)$.
\end{lemma}

\begin{proof}
First of all we consider the conformal diffeomorphism $\phi_{-p_0}$ and notice that $\phi_{-p_0}(p_0)=0$. Therefore, we can consider without loss of generality points
\[
\big\{q_0, \ldots,q_M\big\}\qquad q_j:=\phi_{-p_0}(p_j), \quad q_0=0.
\]
We will show that the points \(\{\phi_{\alpha}(q_j)\}\) are in general position with respect to $\alpha=\phi_{\alpha}(0)$, provided
\(\alpha\) is small enough and
\begin{gather*}
\mathrm{Arg}(\alpha) \neq
\frac{\mathrm{Arg}(q_j) + \mathrm{Arg}(q_\ell)}{2} +k\pi,\qquad \forall j, \ell=1,...,M,\quad \forall k\in\zzz.
\end{gather*}
We now verify the claim. By continuity, we have
\(
\phi_{\alpha}(z)\to z\) for $\alpha\to 0$. Hence, there exists \(\delta_0>0\) such that for $|\alpha|<\delta_0$ any triple of non-aligned points in \(\{q_0, \ldots, q_M\}\) remains non-aligned and, similarly,
\[
|q_j-q_0|\neq |q_\ell-q_0|\quad \Longrightarrow \quad |\phi_\alpha(q_j)-\phi_\alpha(q_0)|\neq |\phi_\alpha(q_\ell)-\phi_\alpha(q_0)|.
\]
Furthermore, if there were two points aligned with $0$, i.e. there exists a line \(r\) such that \(q_j,q_\ell,0\in r\), then since $\mathrm{Arg}(\alpha)\neq \mathrm{Arg}(q_j)+k\pi$ it follows \(\phi_{\alpha}(r)\) is a circle and therefore the points
$\phi_{\alpha}(q_j)$ and  $\phi_{\alpha}(q_\ell)$ are not any more aligned with $\phi_{\alpha}(0)$.

Finally, recall that the circles with center the origin are sent by $\phi_\alpha$ to circles with the center on the open segment between $0$ and $\alpha$. Therefore, if two points $q_j, q_\ell$ belong to the same circle around the origin (i.e., have the same distance from $0$), their images are on an other circle whose center is not $\alpha$. The only possibility in order to have the same distance form $\alpha$ is that $q_j$ and $q_\ell$ are symmetric with respect to the line $r=\{t\, \alpha: t\in \erre\}$. But this instance is excluded by the fact that
\[
\mathrm{Arg}(q_j)\neq \frac{\mathrm{Arg}(q_j) + \mathrm{Arg}(q_\ell)}{2} +k\pi, \quad k\in \zzz.
\]
The claim is then proven and we find a conformal diffeomorphism of the disk $\phi_{-p_0+\alpha} = \phi_\alpha \circ \phi_{-p_0}$ which sends $\{p_0, \ldots, p_M\}$ in general position with respect to $\phi_{-p_0+\alpha}(p_0)=\alpha$.
\end{proof}

\subsection{Proof of Theorem \ref{t:main}}
We are now ready to give the proof of the main theorem.
We show that the sets
\[
\mathcal U_r^{\curlyvee} := \Big\{U\in \cU : m_U(z) = 3\quad 
\forall z\in \cC_U\cap \overline \D_r \Big\},\quad r\in (0,1),
\]
are open dense set in $\overline\cU$.
Therefore, the set $\mathcal O$ of functions $U$ with only triple 
junctions
\[
\mathcal{U}^{\curlyvee}:=\bigcap_{r\in (0,1)} {\mathcal{U}}_r^{\curlyvee}
\]
is residual in $\overline\cU$.

\medskip

\textbf{Openness of ${\mathcal{U}}_r^{\curlyvee}$.}
Let $U\in {\mathcal{U}}_r^{\curlyvee}$ be fixed and assume that $U$ has $M$ 
critical points in $\overline \D_r$ which are triple points:
\[
\mathcal C_U \cap \overline \D_r = \big\{z_1, \ldots, z_{M}\}, 
\quad m_U(z_i)=3 \quad \forall\; i=1, \ldots, M.
\]
We show that there exists $\eps>0$ such that for every $V\in \mathcal{U}$ with $\|U-V\|_{H^{1}(\D)}\leq \eps$ we have that
$V\in \mathcal O_r$, i.e. $V$ has only triple points in $\overline \D_r$. We start observing that
\[
\|f_U-f_V\|_{L^1(\D)} = \|U_z^2 - V_z^2\|_{L^1(\D)} \leq C \|U-V\|_{H^{1}(\D)}\leq C \eps,
\]
for a dimensional constant $C>0$.
On the other hand, since the functions $f_U, f_V$ are holomorphic in $\D$, from the Cauchy formula we deduce that for every $s<1$
there exists a constant $C(s)>0$ such that
\begin{equation}\label{e:vicinanza}
\|f_U-f_V\|_{C^1(\D_{s})}\leq C(s) \|f_U-f_V\|_{L^1(\D)}\leq C(s)\eps.
\end{equation}
We consider a radius $r<s<1$ such that $\mathcal C_U\cap \D_s\subset\cC_U\cap \overline \D_r$. We show that, if $\eps>0$ is sufficiently small, then $\cC_V\cap \D_{\frac{r+s}{2}}$ is all made of triple junctions too.
Indeed, we can consider disjoint circles $C_j=\partial\D_{s_j}(z_j)\subset \D_{\frac{r+s}{2}}$ (oriented counterclockwise) around every $z_j$ for $j=1, \ldots, M$, for a suitable $s_j>0$. Then, by the argument principle we have that
\[
\frac{1}{2\pi i}\oint_{C_j} \frac{f'_U(x)}{f_U(z)}\, dz = 1.
\]
By \eqref{e:vicinanza}, if $\eps>0$ is small enough, then
\[
\left\vert\frac{1}{2\pi i}\oint_{C_j} \frac{f'_V(x)}{f_V(z)}\, dz -
\frac{1}{2\pi i}\oint_{C_j} \frac{f'_V(x)}{f_V(z)}\, dz
\right\vert
\leq C\, \eps,
\]
where the constant $C$ depends on $s<1$ and on $f_U$ (through the 
choice of the circles $C_j$ and the norm of $\|f_U\|_{C^1(\D_s)}$).
Since $\frac{1}{2\pi i}\oint_{C_j}\frac{f'_V(x)}{f_V(z)} dz\in \enne$, 
we conclude that
\[
\frac{1}{2\pi i}\oint_{C_j} \frac{f'_V(x)}{f_V(z)}\, dz=1
\qquad \forall\; j=1, \ldots, M,
\]
i.e., by the argument principle, $f_V$ has a simple zero in 
$\D_{s_j}(z_j)$. Moreover, since $|U|\geq c>0$ in 
$\D_{\frac{r+s}{2}}\setminus \bigcup_{j=1}^M \D_{s_j}(z_j)$ and 
$\|U-V\|_{L^\infty(\D)} =o(1)$ for $\eps \to 0^+$, we also infer that
there are no other critical points of $V$ in $\D_{\frac{r+s}{2}}$, 
thus concluding the proof.


\medskip

\textbf{Density of ${\mathcal{U}}_r^{\curlyvee}$.}
By definition $\cU\subset\overline\cU$ is dense. By Lemma \ref{l.density} the set of functions $U\in \cU$ with 
finitely many critical points is dense. Therefore, it is enough 
to show the density of ${\mathcal{U}}_r^{\curlyvee}$ in such set. We fix then 
$U\in \mathcal{U}$ with finitely many critical point and we 
define the index
\[
\alpha_U:=\sum_{z\in\mathcal{C}_U} \big(m_U(z)-3\big).
\]
Note that $U$ has only triple junctions if and only if $\alpha_U=0$.
We show that, given any $U\in \mathcal{U}$ with finitely many 
critical points and $\alpha_U>0$, for every $\eps>0$ there exists 
$U_\eps\in \mathcal{U}$ with
\begin{equation}\label{e:induzione}
\alpha_{U_\eps}=\alpha_U-1 \qquad\textup{and}\qquad 
\|U-U_\eps\|_{H^1(\D)}\leq \eps.
\end{equation}
Iterating the argument $\alpha_U$ times, we find a segregated 
state $V\in \mathcal{U}$ with $\alpha_V=0$, i.e. with only triple 
junctions, and $\|U-V\|_{H^1(\D)}\leq \alpha_U\eps$. By the 
arbitraryness of $\eps$, we hence conclude.

\medskip

In order to prove \eqref{e:induzione}, let
\[
\mathcal{C}_U = \big\{z_0, \ldots, z_{\alpha_U-1}\big\},
\]
and assume that $\textup{ord}(f_U; z_0)=m_0+1$, with $m_0>0<$.
By Lemma \ref{l.general position} we can consider a conformal map 
$\phi:\D\to \D$ such that $\phi(\mathcal{C}_U)$ is in general
position with respect to $\phi(z_0)$. Then, we can apply Lemma
\ref{l.perturbazione} to $U\circ \phi$ and, for every $\eps_0>0$, we find $V\in \cU$
with $\|U\circ \phi - V\|_{H^1(\D)}\leq \eps_0$ and
\[
\mathcal{C}_{V} = \big\{\omega_0, \phi(z_0), \ldots, 
\phi(z_{\alpha_U-1})\big\},
\]
with 
\[
\textup{ord}(f_V; \omega_0)=1,
\quad\textup{ord}(f_V; \phi(z_0))=m_0
\]
and 
\[ 
\textup{ord}(f_V; \phi(z_i))=\textup{ord}(f_U; \phi(z_i))\quad \forall\; i=1, \ldots, \alpha_U-1.
\]
Then, we conclude that $U_\eps:=V\circ \phi^{-1}\in\cU$ with
\[
\|U - U_\eps\|_{H^1(\D)}\leq
C \|U\circ \phi - V\|_{H^1(\D)}\leq C\eps_0,
\]
for a constant $C>0$ depending on conformal diffeomorphism 
$\phi$, and hence on $U$. By the arbitrariness of $\eps_0$, the proof is concluded.\qed

\begin{remark}
The set ${\mathcal{U}}^{\curlyvee}$ is not open in $\cU$. If $U\in {\mathcal{U}}^{\curlyvee}$ 
extends to a holomorphic function in a neighborhood of $\overline\D$ 
and has a critical point with multiplicity higher than $3$ on the
boundary of $\D$, then every neighborhood of $U$ in $\cU$ cannot be 
contained  in ${\mathcal{U}}^{\curlyvee}$: e.g., any traslation moving the critical 
point from the boundary to the interior produces a function 
$V\not\in{\mathcal{U}}^{\curlyvee}$.
\end{remark}

\section*{Acknowledgement}
We would like to thank Giovanni Alessandrini for his insightful comments.  A source of inspiration for our work.

We would like to thank the referee for various suggestions that helped us to improve the presentation and clarify several technical aspects of the paper.

\section*{Funding}
F.L., E.M. and E.S. have been partially supported by Ateneo 
Sapienza n. prot. RM122181668A9AF5 ``Proprietà qualitative di 
soluzioni di sistemi ellittici singolari: funzioni multi-valore, 
sistemi per la segregazione di specie e problemi a frontiera 
libera''.\\
F. L.  acknowledges the support from the project ”Perturbation 
problems and asymptotics for elliptic differential equations: 
variational and potential theoretic methods” funded by the 
European Union - Next Generation EU and by MUR Progetti di 
Ricerca di Rilevante Interesse Nazionale (PRIN) Bando 2022 
grant 2022SENJZ3.\\
E. S. has been partially supported by PRIN 2022PJ9EFL "Geometric 
Measure Theory: Structure of Singular Measures, Regularity Theory 
and Applications in the Calculus of Variations" funded by the 
European Union -  Next Generation EU, Mission 4 Component 2 - 
CUP: E53D23005860006.

\end{document}